\documentclass[11pt]{amsart}
\usepackage{amsmath,amssymb,amsthm,color,hyperref}
\usepackage{a4wide}
\usepackage[utf8]{inputenc}
\usepackage[T1]{fontenc}
\DeclareFontFamily{OT1}{ptm}{}
\DeclareFontShape{OT1}{ptm}{m}{n} { <-> ptmr}{}
\DeclareFontShape{OT1}{ptm}{m}{it}{ <-> ptmri}{}
\DeclareFontShape{OT1}{ptm}{m}{sl}{ <->ptmro}{}
\DeclareFontShape{OT1}{ptm}{m}{sc}{ <-> ptmrc}{}
\DeclareFontShape{OT1}{ptm}{b}{n} { <-> ptmb}{}
\DeclareFontShape{OT1}{ptm}{b}{it}{ <-> ptmbi}{}     
\DeclareFontShape{OT1}{ptm}{bx}{n} {<->ssub * ptm/b/n}{}
\DeclareFontShape{OT1}{ptm}{bx}{it}{<->ssub * ptm/b/it}{}

\DeclareSymbolFont{bold}{OT1}{ptm}{b}{n}
\DeclareMathAlphabet{\mathbf}{OT1}{ptm}{b}{n}  
\DeclareMathAlphabet{\mathrm}{OT1}{ptm}{m}{n}

\DeclareFontFamily{OT1}{psy}{}      
\DeclareFontShape{OT1}{psy}{m}{n}{ <-> s * [0.9] psyr}{}
\DeclareFontFamily{OMS}{ptm}{}     
\DeclareFontShape{OMS}{ptm}{m}{n}{ <8> <9> <10> gen * cmsy }{}
\DeclareFontFamily{OMS}{cmtt}{}     
\DeclareFontShape{OMS}{cmtt}{m}{n}{ <8> <9> <10> gen * cmsy }{}

\SetSymbolFont{operators}{normal}{OT1}{ptm}{m}{n}   
\SetSymbolFont{operators}{bold}{OT1}{ptm}{b}{n}     
\DeclareSymbolFont{emsy}{OT1}{ptm}{m}{it}
\DeclareSymbolFont{emsr}{OT1}{ptm}{m}{n}
\DeclareSymbolFont{emcmr}{OT1}{cmr}{m}{n}   
\DeclareSymbolFont{emsymb}{OT1}{psy}{m}{n}  
\DeclareMathSymbol a{\mathalpha}{emsy}{"61}
\DeclareMathSymbol b{\mathalpha}{emsy}{"62}
\DeclareMathSymbol c{\mathalpha}{emsy}{"63}
\DeclareMathSymbol d{\mathalpha}{emsy}{"64}
\DeclareMathSymbol e{\mathalpha}{emsy}{"65}
\DeclareMathSymbol f{\mathalpha}{emsy}{"66}
\DeclareMathSymbol g{\mathalpha}{emsy}{"67}
\DeclareMathSymbol h{\mathalpha}{emsy}{"68}
\DeclareMathSymbol i{\mathalpha}{emsy}{"69}
\DeclareMathSymbol j{\mathalpha}{emsy}{"6A}
\DeclareMathSymbol k{\mathalpha}{emsy}{"6B}
\DeclareMathSymbol l{\mathalpha}{emsy}{"6C}
\DeclareMathSymbol m{\mathalpha}{emsy}{"6D}
\DeclareMathSymbol n{\mathalpha}{emsy}{"6E}
\DeclareMathSymbol o{\mathalpha}{emsy}{"6F}
\DeclareMathSymbol p{\mathalpha}{emsy}{"70}
\DeclareMathSymbol q{\mathalpha}{emsy}{"71}
\DeclareMathSymbol r{\mathalpha}{emsy}{"72}
\DeclareMathSymbol s{\mathalpha}{emsy}{"73}
\DeclareMathSymbol t{\mathalpha}{emsy}{"74}
\DeclareMathSymbol u{\mathalpha}{emsy}{"75}
\DeclareMathSymbol v{\mathalpha}{emsy}{"76}
\DeclareMathSymbol w{\mathalpha}{emsy}{"77}
\DeclareMathSymbol x{\mathalpha}{emsy}{"78}
\DeclareMathSymbol y{\mathalpha}{emsy}{"79}
\DeclareMathSymbol z{\mathalpha}{emsy}{"7A}
\DeclareMathSymbol A{\mathalpha}{emsy}{"41}
\DeclareMathSymbol B{\mathalpha}{emsy}{"42}
\DeclareMathSymbol C{\mathalpha}{emsy}{"43}
\DeclareMathSymbol D{\mathalpha}{emsy}{"44}
\DeclareMathSymbol E{\mathalpha}{emsy}{"45}
\DeclareMathSymbol F{\mathalpha}{emsy}{"46}
\DeclareMathSymbol G{\mathalpha}{emsy}{"47}
\DeclareMathSymbol H{\mathalpha}{emsy}{"48}
\DeclareMathSymbol I{\mathalpha}{emsy}{"49}
\DeclareMathSymbol J{\mathalpha}{emsy}{"4A}
\DeclareMathSymbol K{\mathalpha}{emsy}{"4B}
\DeclareMathSymbol L{\mathalpha}{emsy}{"4C}
\DeclareMathSymbol M{\mathalpha}{emsy}{"4D}
\DeclareMathSymbol N{\mathalpha}{emsy}{"4E}
\DeclareMathSymbol O{\mathalpha}{emsy}{"4F}
\DeclareMathSymbol P{\mathalpha}{emsy}{"50}
\DeclareMathSymbol Q{\mathalpha}{emsy}{"51}
\DeclareMathSymbol R{\mathalpha}{emsy}{"52}
\DeclareMathSymbol S{\mathalpha}{emsy}{"53}
\DeclareMathSymbol T{\mathalpha}{emsy}{"54}
\DeclareMathSymbol U{\mathalpha}{emsy}{"55}
\DeclareMathSymbol V{\mathalpha}{emsy}{"56}
\DeclareMathSymbol W{\mathalpha}{emsy}{"57}
\DeclareMathSymbol X{\mathalpha}{emsy}{"58}
\DeclareMathSymbol Y{\mathalpha}{emsy}{"59}
\DeclareMathSymbol Z{\mathalpha}{emsy}{"5A}
\DeclareMathSymbol{\bullet}{\mathalpha}{emsymb}{"B7}
\DeclareMathSymbol{\regis}{\mathalpha}{emsymb}{"D2}
\def\Bullet{\leavevmode\unkern{$\m@th\bullet$}\kern.32em\ignorespaces}
\def\Regis{\leavevmode\raise.5ex\hbox{$\m@th\regis$}}
\DeclareMathSymbol +{\mathbin}{emcmr}{`+}
\DeclareMathSymbol ={\mathrel}{emcmr}{`=}  
\DeclareMathSymbol{\Gamma}{\mathalpha}{emcmr}{"00}
\DeclareMathSymbol{\Delta}{\mathalpha}{emcmr}{"01}
\DeclareMathSymbol{\Theta}{\mathalpha}{emcmr}{"02}
\DeclareMathSymbol{\Lambda}{\mathalpha}{emcmr}{"03}
\DeclareMathSymbol{\Xi}{\mathalpha}{emcmr}{"04}
\DeclareMathSymbol{\Pi}{\mathalpha}{emcmr}{"05}
\DeclareMathSymbol{\Sigma}{\mathalpha}{emcmr}{"06}
\DeclareMathSymbol{\Upsilon}{\mathalpha}{emcmr}{"07}
\DeclareMathSymbol{\Phi}{\mathalpha}{emcmr}{"08}
\DeclareMathSymbol{\Psi}{\mathalpha}{emcmr}{"09}
\DeclareMathSymbol{\Omega}{\mathalpha}{emcmr}{"0A}
\DeclareMathSizes{7.6}{8}{6}{5}

\newcommand{\bigfigure}[1]{{#1}}
\newcommand{\secret}[1]{}
\newcommand{\verlab}[1]{}
\newcommand{\arrlab}[1]{}

\usepackage{tikz}
\usetikzlibrary{cd}  
\tikzset{
dot/.style = {circle, fill, minimum size=#1,
              inner sep=0pt, outer sep=0pt},
dot/.default = 4pt 
}
\tikzset{
bigdot/.style = {circle, draw=black, minimum size=#1,
                 inner sep=0pt}, 
bigdot/.default = 6pt 
}
\usetikzlibrary{braids}

\def\triv{{\rm(\raisebox{-0.5mm}{\includegraphics[width=3mm]{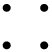}}\rm)}}
\def\west{{\rm(\raisebox{-0.5mm}{\includegraphics[width=3mm]{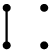}}\rm)}}
\def\south{{\rm(\raisebox{-0.5mm}{\includegraphics[width=3mm]{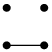}}\rm)}}
\def\east{{\rm(\raisebox{-0.5mm}{\includegraphics[width=3mm]{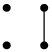}}\rm)}}
\def\north{{\rm(\raisebox{-0.5mm}{\includegraphics[width=3mm]{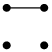}}\rm)}}
\def\sw{{\rm(\raisebox{-0.5mm}{\includegraphics[width=3mm]{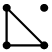}}\rm)}}
\def\nw{{\rm(\raisebox{-0.5mm}{\includegraphics[width=3mm]{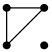}}\rm)}}
\def\se{{\rm(\raisebox{-0.5mm}{\includegraphics[width=3mm]{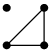}}\rm)}}
\def\ne{{\rm(\raisebox{-0.5mm}{\includegraphics[width=3mm]{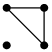}}\rm)}}
\def\ns{{\rm(\raisebox{-0.5mm}{\includegraphics[width=3mm]{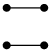}}\rm)}}
\def\ew{{\rm(\raisebox{-0.5mm}{\includegraphics[width=3mm]{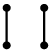}}\rm)}}
\def\mdiag{{\rm(\raisebox{-0.5mm}{\includegraphics[width=3mm]{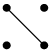}}\rm)}}
\def\adiag{{\rm(\raisebox{-0.5mm}{\includegraphics[width=3mm]{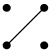}}\rm)}}
\def\square{{\rm(\raisebox{-0.5mm}{\includegraphics[width=3mm]{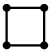}}\rm)}}

\newtheorem{theorem}{Theorem}[section]  

\newtheorem{lemma}[theorem]{Lemma}
\newtheorem{proposition}[theorem]{Proposition}

\newtheorem{conjecture}[theorem]{Conjecture}
\newtheorem{mainconjecture}[theorem]{Main Conjecture}

\theoremstyle{definition}
\newtheorem{definition}[theorem]{Definition}

\newtheorem{remark}[theorem]{Remark}

\newtheorem{example}[theorem]{Example}
\newtheorem{observation}[theorem]{Observation}

\newtheorem{notation}[theorem]{Notation}

\def\black#1{\textcolor{black}{#1}}

\def\Z{{\mathbb{Z}}}

\def\calB{\mathcal{B}}

\def\calR{\mathcal{R}}
\def\co{{\colon\thinspace}}

\def\<{\langle}
\def\>{\rangle}

\def\-{\setminus}

\renewcommand{\phi}{{\varphi}}
\newcommand{\intv}[1]{[\![#1]\!]} 

\begin{document}
\title[Conjugacy invariants in Garside groups]{Conjugacy invariants and rigidity in Garside groups:\\ a uniformity phenomenon}
\author{Matthieu Calvez} 
\address{Matthieu Calvez, IRISIB, Rue Royale 150, 1000 Bruxelles, Belgique}
\email{calvez.matthieu@gmail.com}
\author{Owen Garnier}
\address{Owen Garnier, Instituto de Matem\'aticas de la Universidad de Sevilla (IMUS) and Departamento de \'Algebra, Universidad de Sevilla. Av. Reina Mercedes s/n. 41012, Sevilla.  Spain.}
\email{owen.garnier@math.cnrs.fr}
\author{Juan González-Meneses} 
\address{Juan Gonz\'alez-Meneses, Instituto de Matem\'aticas de la Universidad de Sevilla (IMUS) and Departamento de \'Algebra, Universidad de Sevilla. Av. Reina Mercedes s/n. 41012, Sevilla,  Spain.}
\email{meneses@us.es}
\author{Bert Wiest}
\address{Bert Wiest, Univ Rennes, CNRS, IRMAR - UMR 6625, F-35000 Rennes, France}
\email{bertold.wiest@univ-rennes.fr}
\begin{abstract}
Consider an element~$x$ of a Garside group which is rigid in the sense of Garside-theory. Let $SC(x)$ be the set of rigid conjugates of~$x$ -- this is a well-known characteristic subset of the conjugacy class of~$x$. 
We present computational evidence that the sequence $\left( |SC(x^n)| \right)_{n\in\mathbb N}$ is not only bounded, but in fact periodic, and that there is a bound on the length of the period which depends only on the underlying group and its Garside structure.
We prove this result in the special case of the circular Garside groups, including the 2-generator Artin groups with their classical and dual structures (where we prove that the sequence is always constant), and in the case of the dual $4$-strand braid group.
\end{abstract}
\maketitle

%
\section{Introduction}\label{S:Intro}

Let $G$ be a Garside group, equipped with a Garside structure -- 
the most well-known examples are the $m$-strand braid group equipped with Garside's original Garside structure, denoted~$\calB_m$, and the same group equipped with Birman-Ko-Lee's dual Garside structure, denoted~$\calB^*_m$. Other examples relevant for the present paper are the Artin-Tits groups of spherical type, equipped with their classical Garside structures and the so-called \emph{circular Garside groups}. 
Relevant references for the braid groups with the classical structure are \cite{Garside, ElRifaiMorton} and \cite{BKL} for the dual structure, \cite{BrieskornSaito,Deligne} for the generalisation to Artin-Tits groups of spherical type. The circular groups are studied in~\cite{GarnierRoots}. For general Garside theory, see \cite{DehornoyParis} (for the first definition), as well as \cite{DehornoyGarside,Birman-Gebhardt-GM1}, and \cite{McCammondIntroGarside}, and finally \cite{DDGKM} for a very complete and high-level modern point of view.

We recall that to any element $x$ of~$G$, we can associate a finite subset of the conjugacy class of~$x$ called the sliding circuits set, denoted $SC(x)$ \cite{GebhardtGMCyclicSliding}. This subset depends only on the conjugacy class of~$x$; therefore, being able to calculate the sliding circuits set of any given element solves the conjugacy problem in~$G$. 

Every element $x$ in a Garside group admits a {\em Garside normal form}, which is a decomposition of the form $x=\Delta^p x_1\cdots x_r$, where $\Delta$ is the Garside element and $x_1,\ldots,x_r$ belong to a set of generators called {\em simple elements}, which are positive prefixes of $\Delta$. The sequence $x_1,\ldots,x_r$ (which must satisfy some algebraic conditions) and the integer $p$ (called the {\em infimum} of $x$ and denoted $\inf(x)$) are determined by $x$, so the Garside normal form of $x$ is unique. 

It is a property of Garside groups that conjugation by $\Delta$ (we will denote $\tau(s)=\Delta^{-1}s\Delta$) preserves the set of simple elements. Now, an element $x$ with normal form $\Delta^p x_1\cdots x_r$ is said to be {\em rigid} if the Garside normal form of $x^2$ is just $\Delta^{2p} \tau^{p}(x_1)\cdots \tau^p(x_r) x_1 \cdots x_r$. That is, if the normal form of $x^2$ is obtained by concatenating two normal forms of $x$, and then sliding the middle $\Delta^p$ to the left, conjugating the factors on its left. 

Notice that, if $x$ is rigid, every positive power of $x$ is also rigid: the left normal form of $x^n$ is obtained by concatenating $n$ normal forms of $x$ and sliding all $\Delta$ factors to the left. Moreover, since the Garside normal form of $x^{-1}$ is closely related to that of $x$, it follows that $x$ is rigid if and only if $x^{-1}$ is rigid, so the rigidity of $x$ implies the rigidity of every (non-necessarily positive) power of $x$.

If $x$ is rigid, then its sliding circuits set $SC(x)$ consists of the set of rigid conjugates of $x$\cite{GebhardtGMCyclicSliding}. Hence, $SC(x^n)$ consists of the set of rigid conjugates of $x^n$, for every integer $n$. It follows that set $SC(x^{-n})$ consists of the inverses of the elements of $SC(x^n)$. We will be interested in the sizes of these sets, as $n>0$ grows.

\begin{conjecture}\label{C:BasicVersion}
For any rigid element $x$, the sequence $\left( |SC(x^n)| \right)_{n\in\mathbb N}$ is bounded.
\end{conjecture}

There are at least two motivations for studying this conjecture.
\begin{enumerate}
\item Currently the biggest roadblock to solving the conjugacy problem in braid groups in polynomial time using Garside theory is the following open question, concerning pseudo-Anosov (pA) braids: how quickly can 
$$\max_{x\in \calB_m, |x|\leqslant \ell,x\mathrm{pA}}|SC(x)|$$
grow, as a function of $\ell$ (for any fixed value of~$m$)? In particular, it is open whether or not this function is polynomially bounded. It would be valuable to at least understand how the size of the sliding circuit set of powers of a single element behaves: we expect it to remain bounded.
\item Let us look at the Cayley graph of~$G/\Delta^e$ (where $\Delta^e$ is a central power of the so-called Garside element~$\Delta$), with respect to the generating set given by the Garside structure. 
Let $x$ be a rigid element of~$G$; for the purposes of this discussion, we add the simplifying hypothesis that $\inf(x)=0$ (see~\cite{CalvezWiestMorse} for a discussion of the general case).
In this Cayley graph, the rigid element~$x$ gives rise to a bi-infinite geodesic axis traversing the identity vertex.
Any rigid conjugate of~$x^n,$ for any~$n$, yields another geodesic axis at finite Hausdorff-distance from the axis of~$x$. 

Thus if the sequence $(|SC(x^n)|)_{n\in\mathbb N}$ is not bounded, then there are infinitely many distinct parallel axes. In particular, if the axis of~$x$ has the Morse property, then the sequence $(|SC(x^k)|)_{k\in\mathbb N}$ must be bounded. 

In the special case where $G$ is a braid group, it is known that pseudo-Anosov elements have the Morse property. Thus we know already that for any rigid pseudo-Anosov element~$x$ in a braid group, Conjecture~\ref{C:BasicVersion} holds: the sequence $(|SC(x^n)|)_{n\in\mathbb N}$ is bounded. Nevertheless, if we want to understand the geometry of the ``bundle'' of axes parallel to the axis of~$x$, we need study the sequence in more detail. 
\end{enumerate}

While performing extensive computer experiments in order to test Conjecture~\ref{C:BasicVersion}, the authors were surprised to discover that a much stronger result appears to hold. Indeed, there seems to be a uniformity phenomenon: for any fixed number of strands, the power for which the sliding circuit set takes its maximal size appears to be bounded:

\begin{mainconjecture}\label{C:MainConjecture}
For any Garside group~$G$, equipped with a Garside structure, there exists a finite set of integers~$\mathcal P_G$ with the following property: for any rigid element $x$ of~$G$, the sequence $(|SC(x^n)|)_{n\in\mathbb N}$ is periodic, and the period is an element of $\mathcal P_G$. 
\end{mainconjecture}

This conjecture is the fruit of large computer experiments with various Artin-Tits groups of spherical type. Large numbers of random elements were generated (by writing random words in Artin generators), and for those that happened to have a rigid conjugate, the sizes of the sliding circuit sets of the braids and their first few powers were calculated. 

\begin{remark} Concerning the computer programs used for our calculations:
for the braid groups equipped with the classical Garside structure, the calculations were based on the C++ program Cbraid~\cite{GMcbraid}. 
By contrast, GAP3 with the development version of the Chevie package~\cite{GAP3,Chevie,MichelChevie}) was used for braid groups with the dual Garside structure, as well as for other Artin-Tits groups (with the classical Garside structure). In the latter cases, we have much less experimental data, since calculations on GAP3/Chevie were substantially slower than on the C++ program \cite{GMcbraid}. The port to Julia of Chevie \cite{MichelJuliaVersion} is substantially faster than the GAP3 version, but it was not used. Also, after the calculations were completed, Matteo Wei created a profoundly rewritten version \cite{WeiGarCide} of the C++ program, which can deal extremely efficiently with very general Garside groups.
\end{remark}

As a result of our calculations, we have the following table with the conjectured sets of possible periods. Here $\calB_m$ denotes the $m$-strand braid group (which is the same as the Artin-Tits group of type~$A_{m-1}$), equipped with the classical Garside structure; $\calB_m^*$ denotes the same group, equipped with the dual Garside structure. We did not perform calculations with the dual structure on Artin-Tits groups other than the braid groups.
\medskip

\phantom{.}\begin{tabular}{r||c|c|c|c|c|c|}
Group $G$ & $\calB_3$ & $\calB_4$ & $\calB_5$ & $\calB_6$ & $\calB_7$ & $\calB_8$ \\
\hline
\phantom{$\frac{1^1}{1}$}$\mathcal P_G$ & $ \{1\}$ & $\{1,2\}$ & $\{1,2,3\}$ & $\{1,2,3,6\}$ & $\{1,2,3,4,6\}$ & $\{1,2,3,4,6,12\}$\\
\end{tabular}
\smallskip

\phantom{.}\begin{tabular}{r||c|c|c|c}
 Group $G$ & $\calB_3^*$ & $\calB_4^*$ & $\calB_5^*$ & $\calB_6^*$\\
 \hline
 \phantom{$\frac{1^1}{1}$}$\mathcal P_G$ & $\{1\}$ & $\{1,2,3\}$ & $\{1,2,3,4,6\}$
 & $\{1,2,3,4,5,6
 \}$\end{tabular}

\phantom{.}\begin{tabular}{r||c|c|c|c|c|c|c|c|c|}
 Group $G$ & $B_2$ & $B_3, B_4$ & $B_5$ & $D_4$ & $D_5$ & $F_4$ & $H_3$ & $H_4$ & $I_2(10)$ \\
\hline
 \phantom{$\frac{1^1}{1}$}$\mathcal P_G$ & $\{1\}$ & $ \{1,2\}$ & $\{1,2,3\}$ & $\{1,2,3\}$ & $\{1,2,3,6\}$ & $\{1,2,3\}$ & $\{1,2\}$ & $\{1,2,3\}$ & $\{1\}$\\
\end{tabular}
\medskip

\begin{notation}
In the braid group $\calB_m$, we will use the notation $i$ instead of the more standard~$\sigma_i$ for the $i$th Artin half-twist generator. Also, when we write $21|12$, the symbol $|$ indicates that the product of two Garside generators $\sigma_2 \sigma_1 \cdot \sigma_1 \sigma_2$ is in Garside normal form as written. 
\end{notation}

In this paper, we will represent the sets of sliding circuits $SC(x)$, and the simple elements that connect the elements of $SC(x)$ by conjugations, using some graphs that we will call \emph{conjugacy graphs} (see Definition~\ref{D:ConjugacyGraph}). 

The vertices of the conjugacy graph of $x$ correspond to orbits (in $SC(x)$) under the action of to kinds of conjugations called \emph{cycling} and $\tau$. Each vertex will be labeled with a representative of the orbit.

The oriented edges correspond to conjugations by two possible kinds of simple elements: a black arrow (resp. gray arrow), going from a vertex labeled~$y$ to a different vertex labeled~$z$, corresponds to a simple element which is a prefix of the initial factor of~$y$ (resp. of the complement of the final factor of~$y$), conjugating $y$ to an element in the orbit of $z$. See Section~\ref{S:Prelim} for details. If there are exactly $k$ black (resp. gray) arrows conjugating $y$ to an element in the orbit of~$z$, we will draw a single arrow labeled ``$\times k$''.

\begin{example}\label{E:SimpleExInB4classical}
In $\calB_4$, for $x=21|12|2132$, we have $|SC(x^k)|=6$ if $k$ is odd, and $|SC(x^k)|=18$ if $k$ is even.
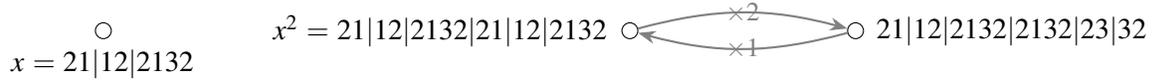
\begin{figure}[htb]
\bigfigure{ \begin{center}\begin{tikzpicture}
\node[bigdot,label={below:$x=21|12|2132$}] (x) at (-7,0) {};
\node[bigdot,label={left:$x^2=21|12|2132|21|12|2132$}] (x2) at (0,0) {};
\node[bigdot,label={right:$21|12|2132|2132|23|32$}] (other) at (3,0) {};
\draw[-Stealth,thick,gray] (x2)  to[bend left=15] node {\small{$\times 2$}} (other);
\draw[-Stealth,thick,gray] (other)  to[bend left=15] node {\small{$\times 1$}} (x2);
\end{tikzpicture}\end{center} }
\caption{For $x=21|12|2132$, the figure shows on the left the conjugacy graph of $x$ (only one vertex), and on the right the conjugacy graph of $x^2$.}
\label{F:ConjGraphB4easy} \end{figure}
Figure~\ref{F:ConjGraphB4easy} shows the conjugacy graphs of~$x$ and of~$x^2$. 

On the left, we see that $SC(x)$ consists only of $x$, of the two elements obtained from~$x$ by cyclically permuting the three factors, and of their images under~$\tau$, so it contains only one vertex.

The set $SC(x^2)$, by contrast, is larger: we can conjugate $x^2$ in two different ways (by~$1$ or by~$3$), and we obtain two elements of $SC(x^2)$ which are cyclic conjugates of each other (so they both represent the right-hand vertex). 
If we want to conjugate back, from one of the elements thus obtained to a rigid braid representing to the left vertex, then we have only one choice: the only minimal simple conjugator that conjugates $1^{-1}\cdot(x^2)\cdot 1=21|12|2132|2132|23|32$ to a cyclic conjugate of $x^2$ is the braid~$3$. The arrows are drawn in gray color, because they represent conjugations by a prefix of the complement of the last factor -- in the notation of \cite{Birman-Gebhardt-GM2}, they represent "gray arrows". (By contrast, ``black arrows'' are conjugations by prefixes of the first factor.)

The next figure shows a part of the Cayley graph, and specifically the conjugation of~$x^2$ by~$1$ yielding $21|12|2132|2132|23|32$. The blue arcs indicate sequences of generators that are in normal form. It is also intuitive from the red arrows in this figure why $y=1^{-1}\cdot x\cdot 1$ is not rigid whereas $1^{-1}\cdot x^2\cdot 1$ is. In Section~\ref{S:Prelim} we will give a detailed explanation of how the calculation of a conjugate along a gray or black arrow works.
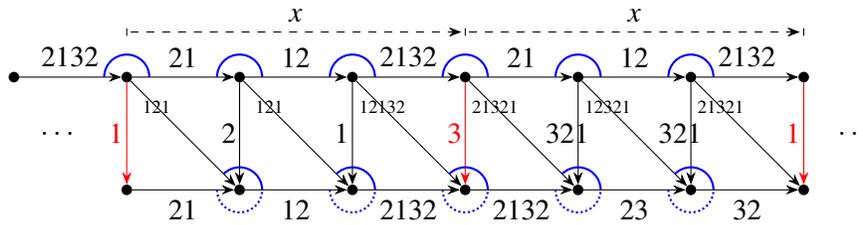
\begin{figure}[htb]
\bigfigure{ \begin{center}\begin{tikzpicture}
\node[dot] (u1) at (0,1.5) {};
\node[dot] (u2) at (1.5,1.5) {};
\node[dot] (u3) at (3,1.5) {};
\node[dot] (u4) at (4.5,1.5) {};
\node[dot] (u5) at (6,1.5) {};
\node[dot] (u6) at (7.5,1.5) {};
\node[dot] (u7) at (9,1.5) {};
\node[dot] (u8) at (10.5,1.5) {};
\node[dot] (d2) at (1.5,0) {};
\node[dot] (d3) at (3,0) {};
\node[dot] (d4) at (4.5,0) {};
\node[dot] (d5) at (6,0) {};
\node[dot] (d6) at (7.5,0) {};
\node[dot] (d7) at (9,0) {};
\node[dot] (d8) at (10.5,0) {};
\draw[-Stealth] (u1) to node[above] {$2132$} (u2); 
\draw[-Stealth] (u2) to node[above] {$21$} (u3);
\draw[-Stealth] (u3) to node[above] {$12$} (u4);
\draw[-Stealth] (u4) to node[above] {$2132$} (u5);
\draw[-Stealth] (u5) to node[above] {$21$} (u6);
\draw[-Stealth] (u6) to node[above] {$12$} (u7);
\draw[-Stealth] (u7) to node[above] {$2132$} (u8);
\draw[-Stealth,red] (u2) to node {$1$\phantom{w}} (d2); 
\draw[-Stealth] (u3) to node {$2$\phantom{w}} (d3);
\draw[-Stealth] (u4) to node {$1$\phantom{w}} (d4);
\draw[-Stealth,red] (u5) to node {$3$\phantom{w}} (d5);
\draw[-Stealth] (u6) to node {$321$\phantom{w}} (d6);
\draw[-Stealth] (u7) to node {$321$\phantom{w}} (d7);
\draw[-Stealth,red] (u8) to node {$1$\phantom{w}} (d8);
\draw[-Stealth] (u2) to node[near start] {\tiny $121$} (d3);
\draw[-Stealth] (u3) to node[near start] {\tiny $121$} (d4);
\draw[-Stealth] (u4) to node[near start] {\tiny $12132$} (d5);
\draw[-Stealth] (u5) to node[near start] {\tiny $21321$} (d6);
\draw[-Stealth] (u6) to node[near start] {\tiny $12321$} (d7);
\draw[-Stealth] (u7) to node[near start] {\tiny $21321$} (d8);
\draw[-Stealth] (d2) to node[below] {$21$} (d3);
\draw[-Stealth] (d3) to node[below] {$12$} (d4);
\draw[-Stealth] (d4) to node[below] {$2132$} (d5);
\draw[-Stealth] (d5) to node[below] {$2132$} (d6);
\draw[-Stealth] (d6) to node[below] {$23$} (d7);
\draw[-Stealth] (d7) to node[below] {$32$} (d8);
\node[circle,label={right:$\ldots$}] (dots1) at (0,0.75) {};
\node[circle,label={right:$\ldots$}] (dots2) at (10.6,0.75) {};
\draw[color=blue,thick] (1.8,1.5) arc [start angle=0, end angle=180, radius=0.3];
\draw[color=blue,thick] (3.3,1.5) arc [start angle=0, end angle=180, radius=0.3];
\draw[color=blue,thick] (4.8,1.5) arc [start angle=0, end angle=180, radius=0.3];
\draw[color=blue,thick] (6.3,1.5) arc [start angle=0, end angle=180, radius=0.3];
\draw[color=blue,thick] (7.8,1.5) arc [start angle=0, end angle=180, radius=0.3];
\draw[color=blue,thick] (9.3,1.5) arc [start angle=0, end angle=180, radius=0.3];
\draw[color=blue,thick] (3.3,0) arc [start angle=0, end angle=135, radius=0.3];
\draw[color=blue,thick] (4.8,0) arc [start angle=0, end angle=135, radius=0.3];
\draw[color=blue,thick] (6.3,0) arc [start angle=0, end angle=135, radius=0.3];
\draw[color=blue,thick] (7.8,0) arc [start angle=0, end angle=135, radius=0.3];
\draw[color=blue,thick] (9.3,0) arc [start angle=0, end angle=135, radius=0.3];
\draw[densely dotted,color=blue,thick] (2.7,0) arc [start angle=180, end angle=360, radius=0.3];
\draw[densely dotted,color=blue,thick] (4.2,0) arc [start angle=180, end angle=360, radius=0.3];
\draw[densely dotted,color=blue,thick] (5.7,0) arc [start angle=180, end angle=360, radius=0.3];
\draw[densely dotted,color=blue,thick] (7.2,0) arc [start angle=180, end angle=360, radius=0.3];
\draw[densely dotted,color=blue,thick] (8.7,0) arc [start angle=180, end angle=360, radius=0.3];
\draw[-Stealth,dashed] (1.55,2.1) to node[above] {$x$} (5.95,2.1);
\draw[-Stealth,dashed] (6.05,2.1) to node[above] {$x$} (10.45,2.1);
\draw[-] (1.5,2.05) to (1.5,2.15);
\draw[-] (6,2.05) to (6,2.15);
\draw[-] (10.5,2.05) to (10.5,2.15);
\end{tikzpicture}\end{center} }
\caption{The conjugation of $x^2$ by~$1$ for $x=21|12|2132$, seen in the Cayley graph. This diagram can be periodically extended into a bi-infinite strip. This type of diagram is called a domino diagram.}
\label{F:ConjugationInB4easy} \end{figure}

Let us reinterpret our conjectures in light of this example. We have seen that a rigid braid~$x$ may very well have a conjugate~$y$ which is not rigid, but whose square (living in $SC(x^2)$) is rigid. At first sight, such an element~$y^2$ doesn't look like a square, as its normal form does not consist of some shorter word that is repeated twice. 
Now it may happen (though it doesn't in this example) that some further, even higher power of~$y$ can be conjugated again, yielding yet more rigid conjugates; and there may even be a whole tower of such powers. However, Conjecture~\ref{C:BasicVersion} states that such a tower is necessarily finite, and Conjecture~\ref{C:MainConjecture} claims that there is a bound on the highest power involved, and this bound is uniform in the group.
\end{example}

\begin{notation}\label{N:SmallestRigidPower}
For any element $y$ of~$G$ which possesses a rigid power, we denote $r(y)$ the smallest positive integer such that $y^{r(y)}$ is rigid.
\end{notation}

\begin{conjecture}(Rigid power implies uniformly small rigid power)\label{C:MainConjecture2} 
For any Garside group~$G$, there exists an integer $\calR_G$ such that for any element $y$ of~$G$ which is conjugate to a rigid element and which has a rigid power, $r(y)\leqslant \calR_G$.
\end{conjecture}

\begin{notation}
We say that in a group $G$ \emph{roots are unique up to conjugacy} if for any two elements $x_1, x_2$ and any positive integer~$k$, the identity $x_1^k=x_2^k$ implies that $x_1$ and $x_2$ are conjugate. For instance, this property is satisfied in braid groups (and roots of pseudo-Anosov elements are actually unique, not just up to conjugacy)~\cite{GMRootUnique}, in Artin-Tits groups of type~$B$~\cite{LeeLeeRoots}, in complex braid groups of rank~2~\cite{GarnierRoots}, and it may well hold in all spherical-type Artin-Tits groups~\cite{DDGKM}(Conjecture X.3.10).
\end{notation}

\begin{proposition}\label{P:ImplicConjectures}
If Conjecture~\ref{C:MainConjecture2} holds in some Garside group~$G$ where roots are unique up to conjugacy, then our Main Conjecture~\ref{C:MainConjecture} holds for the same group~$G$. 
\end{proposition}
Our first main result is that bounded implies periodic:

\begin{theorem}\label{T:LatticeOfExponents}
(a) Let $G$ be a Garside group, and $y$ an element which possesses a rigid power. Then $y^n$ is rigid if and only if $n$ is a multiple of~$r(y)$.

(b) Suppose $G$ is a Garside group where roots are unique up to conjugacy, and
$x$~is a rigid element of~$G$. If the sequence $\left( |SC(x^n)| \right)_{n\in\mathbb N}$ is bounded then it is periodic. 
More precisely, if there is a minimal positive integer~$r_*$ such that $|SC(x^{r_*})| = \sup_{n\in\mathbb N} |SC(x^n)|$, then the sequence is periodic of period~$r_*$.
\end{theorem}

Our second main result is that Main Conjecture~\ref{C:MainConjecture}, is true for some of the most important and basic examples. The definition of the family of \emph{circular groups} from~\cite{GarnierRoots} will be recalled in Section~\ref{S:Proof}, but we mention here already that it contains all the 2-generator Artin groups (type $A_2$, i.e.\ the 3-strand braid group~$\calB_3$, type~$B_2$ and types $I_2(m)$ for $m\geqslant 5$), all of them with their classical and dual Garside structures.

\begin{theorem}\label{T:main}
(a) Conjecture~\ref{C:MainConjecture2}, and hence also the Main Conjecture~\ref{C:MainConjecture}, holds if $G=G(m,\ell)$ is a circular Garside group. 
Namely, $\mathcal P_{G}=\{1\}$ and $\calR_{G}=\frac{m}{m\wedge \ell}$.

(b) Conjecture~\ref{C:MainConjecture} holds for~$\calB_4^*$, the four-strand braid group with the dual Garside structure: namely, $\mathcal P_{\calB_4^*} \subseteq \{1,2,\ldots,lcm(1,\ldots,27)\}$.
\end{theorem}

\begin{remark}
Using a detailed combinatorial analysis, it would be possible to prove that in~$\calB_4^*$, the sequence $|SC(x^n)|$ can only have three possible behaviors: it is either constant, or it is periodic of period~2, or it is periodic of period~3 with $|SC(x^2)|=|SC(x)|$ and $|SC(x^3)|=4\cdot |SC(x)|$. Correspondingly, for any rigid~$x$, either all conjugates $y$ of~$x$ with a rigid power satisfy $r(y)=1$ or $r(y)=2$, or they all satisfy $r(y)=1$ or $r(y)=3$.
We will only prove the weaker result stated in Theorem~\ref{T:main}(b).
\end{remark}

\begin{remark}
There is a potential connection of this work to the world of hierarchically hyperbolic groups~\cite{SistoHHS} and injective spaces \cite{HaettelHodaPetyt}. Indeed, the property of~$x$ not to have new members of $SC(x^n)$ for large values of~$n$ is reminiscent of the strong constriction property, which is equivalent to the strong contraction property~\cite{ArzhantsevaCashenTao}.
Now mapping class groups act on various injective spaces which are quasi-isometric to the Cayley graph, and axes of pseudo-Anosov elements are strongly contracting in these spaces~\cite{SistoZalloum}. An example of such an injective space is the one constructed by Petyt and Zalloum in~\cite{PetytZalloum}. It is reasonable to wonder if an analogue property of our conjectured ``uniform boundedness of $SC(x^n)$'' holds in these injective spaces: in a fixed mapping class group, are pseudo-Anosov axes uniformly strongly contracting, in the sense that there is a constant $C$ (depending only on the group) such that any ball disjoint from the axis of a pseudo-Anosov element~$x$ projects to a part of the axis whose diameter is bounded by $C\cdot \ell(x)$, where $\ell(x)$ denotes the translation length of~$x$? A.~Zalloum has informed the authors~\cite{ZalloumPersonal} that indeed he believes this to be the case.
\end{remark}

The paper is organized as follows: in Section~\ref{S:Prelim} we present the required background material from Garside theory -- much of the material in this section is well-known. Section~\ref{S:Examples} presents many examples which illustrate the subtlety of the problem. In Section~\ref{S:Periodicity} we prove our first main result, Theorem~\ref{T:LatticeOfExponents} and the fact that, for groups where roots are unique up to conjugacy, the truth of Conjecture~\ref{C:MainConjecture2} implies the truth of our Main Conjecture~\ref{C:MainConjecture}. Finally in Section~\ref{S:Proof} we give the proof of our second main result, Theorem~\ref{T:main}, proving our conjecture in some important special cases.


\section{General framework and preliminary results}\label{S:Prelim}

We start with some reminders of general Garside theory -- the reader can consult \cite{Birman-Gebhardt-GM1, Birman-Gebhardt-GM2, DehornoyGarside, DehornoyQuadrNorm, ElRifaiMorton, McCammondIntroGarside}. 
We recall that in a Garside group~$G$, equipped with a Garside structure, every element~$x$ is represented by a unique word in \emph{normal form}, taking the shape $x=\Delta^k\cdot x_1| \ldots |x_\ell$. 
Here $\Delta$ is the particular element given by the Garside structure (e.g. the half-twist braid in $\calB_m$ and the $\frac{2\pi}{m}$-twist braid in $\calB^*_m$); each $x_i$ is a \emph{Garside generator} or a \emph{simple element}, and each pair of successive letters $x_i\cdot x_{i+1}$ is \emph{left-weighted} -- these terms will be defined shortly -- and we indicate this by writing the product $x_i|x_{i+1}$. 
We denote $k=\inf(x)$ the \emph{infimum}, $k+\ell=\sup(x)$ the \emph{supremum}, and $\ell=\ell_{can}(x)$ the \emph{canonical length} of~$x$. The automorphism $\tau\co G\to G$ is defined by $\tau(x)=\Delta^{\!-1} x \Delta$.

We recall the \emph{prefix ordering} $\preccurlyeq$ on $G$, which is a partial ordering given by $g_1\preccurlyeq g_2$ if and only if $g_1^{-1}g_2\in G^+$, where $G^+$ is the submonoid of positive elements of $G$, that is, products of non-negative powers of generators.
This is a lattice ordering: any two elements $g_1$, $g_2$ of~$G$ have a meet $g_1\wedge g_2$, and a join $g_1\vee g_2$. The Garside generators of~$G$, or \emph{simple elements}, are precisely the elements $x$ satisfying $1\preccurlyeq x\preccurlyeq \Delta$. The \emph{complement} $\partial x$ of a simple element~$x$ is $\partial x=x^{-1}\Delta$. A product of two simple elements $x_1\cdot x_2$ is said to be \emph{left-weighted}, or \emph{in normal form} if there is no prefix $p$ of $x_2$ such that $x_1 p$ is simple, or equivalently, if $x_2\wedge \partial x_1=1$; in this case, we write $x_1|x_2$. This completes our definition of the normal form.

For an element $x$ of~$G$ with normal form $\Delta^{\inf(x)} x_1|\ldots|x_\ell$ and $\ell>0$, the \emph{initial} and \emph{final factors} are defined by $\iota(x)=\tau^{-\inf(x)} (x_1)$ and $\tau(x)=x_\ell$. A element~$x$ is said to be \emph{rigid} if either it is a power of~$\Delta$ or if the product $\phi(x)\cdot \iota(x)$ is in normal form as written: $\phi(x)|\iota(x)$. 
Equivalently, saying that $x$ is rigid means that the normal form of~$x^2$ is as expected: $x^2=\Delta^{2\cdot\inf(x)}\tau^{\inf(x)}(x_1)|\ldots |\tau^{\inf(x)}(x_\ell)|x_1|\ldots|x_\ell$.
If $\inf(x)=0$, the idea is that the normal form word representing~$x$, regarded as a cyclic word, is still in normal form. 

If an element $g$ of~$G$ is conjugate to a rigid element, then, following \cite{GebhardtGMCyclicSliding}, we define $SC(g)$ to be the set of all rigid conjugates of~$g$. (In fact, this is not the original definition of~\cite{GebhardtGMCyclicSliding}, but it is proven in this paper that the original definition -- which we won't need -- coincides with the one given here in the special case of rigid elements.)


\subsection{Calculating rigid conjugates}

To any rigid element~$x$, we will associate a graph with colored, oriented edges called the \emph{conjugacy graph}, already used in \cite{CalvezWiestConjugacyB4}, and very similar to the graphs introduced earlier in \cite{Birman-Gebhardt-GM2}. 

\begin{definition}\label{D:ConjugacyGraph}
Let $x$ be a rigid element of~$G$. The \emph{conjugacy graph} of~$x$ has one vertex for every orbit in $SC(x)$ under the action of \emph{cycling} (i.e.conjugation by~$\iota(x)$) and of $\tau$ (conjugation by~$\Delta$). 
The edges, which we will call \emph{arrows}, come in two colors: black and gray. A vertex represented by an element~$y$ is connected by a \emph{black arrow} to a vertex represented by~$z$ if there is a prefix $c\preccurlyeq \iota(y)$ which conjugates $y$ to an element in the orbit of $z$, i.e.if $c^{-1}y c$ belongs to the orbit of $z$. By contrast, there is a \emph{gray arrow} between these two vertices if the conjugator~$c$ satisfies $c\preccurlyeq \partial \phi(y)$. 
Any arrow can carry a label of the form ``$\times k$'' (with $k\geqslant 2$), indicating that there are $k$ different conjugations of one representative of the source vertex, yielding various representatives of the target vertex.
\end{definition}

One can easily check that $\tau$ and cycling commute. If $y$ is rigid, every element in the orbit of $y$ under cycling and $\tau$ has the form $\Delta^p \tau^k(y_j)|\cdots|\tau^k(y_r)| \tau^{k-p}(y_1)|\cdots| \tau^{k-p}(y_{j-1})$ for some $j\in \{1,\ldots, r\}$ and some integer $k$. Hence, if $e$ is the smallest positive integer such thast $\Delta^e$ is central (in braid groups with the classic Garside structure $e=2$, and with the dual Garside structure $e=n$), then the size of the orbit of $y$ is at most $er$. 

One can also check that, if there are exactly $k$ black (resp. gray) arrows conjugating an element $y$ to an element in the orbit of $z$, then, for every $y'$ in the orbit of $y$, there will precisely be $k$ black (resp. gray) arrows conjugating $y'$ to an element in the orbit of $z$. Hence, the arrows in the conjugacy graph (and their multiplicities) are well defined.

\begin{remark}\label{R:DiffDefnsConjGraph}
Our conjugacy graphs are almost the same as the graphs introduced in \cite{Birman-Gebhardt-GM2}, but there are some subtle differences: firstly, in the graphs of \cite{Birman-Gebhardt-GM2}, the vertices are orbits of $SC(x)$ under cycling, but not under the action of~$\tau$.
Secondly, the graphs of \cite{Birman-Gebhardt-GM2} have fewer arrows per vertex than our graphs, because they only contain minimal arrows (i.e.arrows that are not given by the concatenation of two or more arrows of the same color). 

It follows from the results in~\cite{Birman-Gebhardt-GM2} that conjugacy graphs are always connected, and that there is a path from any vertex to any other vertex, first along a sequence of gray arrows, and then along a sequence of black arrows (or vice versa).
\end{remark}

\begin{remark}\label{R:DominoVsConjugacy}
In this paper, we are drawing two types of diagrams, which should not be confused. 
On the one hand, we have \emph{domino diagrams} (Figures~\ref{F:ConjugationInB4easy}, \ref{F:DominoDiagramInf0}, \ref{F:DominoDiagramInfNeq0}, \ref{F:ConjugationInB4dPer2}, \ref{F:ConjugationInB4dWithInf}, \ref{F:ConjugationInB4dPer3}), which live in the Cayley graph of~$G$ (with generators $=$ simple elements). In domino diagrams, arrows indicate right multiplication by a simple element, and the little blue arcs connecting the end of one arrow $x_i$ to the start of another arrow $x_{i+1}$ indicate that the word $x_i\cdot x_{i+1}$ is in normal form. 
On the other hand, we have have pictures of \emph{conjugacy graphs} (Figures~\ref{F:ConjGraphB4easy}, \ref{F:TheBigOne}, \ref{F:ConjGraphB4dPer2}) from Definition~\ref{D:ConjugacyGraph}, where (black or gray) arrows indicate conjugations by simple elements.
\end{remark}

Next, we recall the standard method for calculating the conjugate of a rigid braid $x=\Delta^k x_1|\ldots|x_\ell$ (with $\ell>0$) along a gray arrow, i.e.calculating the normal form of $c^{-1}xc$ for some simple element $c$ such that $x_\ell\cdot c\preccurlyeq \Delta$. 
We will see that this is essentially the same problem as calculating the normal form of~$x_\ell\cdot x\cdot c$ for such an element~$c$. 
We will first suppose, for simplicity, that $\inf(x)=0$; the process is illustrated by the domino diagram in Figure~\ref{F:DominoDiagramInf0}, in the special case $\ell=3$. For an explicit example, see Figure~\ref{F:ConjugationInB4easy}.

\begin{figure}[htb]
\bigfigure{ \begin{center}\begin{tikzpicture}
\node[dot] (u1) at (0,1.5) {};
\node[dot] (u2) at (1.5,1.5) {};
\node[dot] (u3) at (3,1.5) {};
\node[dot] (u4) at (4.5,1.5) {};
\node[dot] (u5) at (6,1.5) {};
\node[dot] (d2) at (1.5,0) {};
\node[dot] (d3) at (3,0) {};
\node[dot] (d4) at (4.5,0) {};
\node[dot] (d5) at (6,0) {};
\draw[-Stealth] (u1) to node[above] {$x_3$} (u2); 
\draw[-Stealth] (u2) to node[above] {$x_1$} (u3);
\draw[-Stealth] (u3) to node[above] {$x_2$} (u4);
\draw[-Stealth] (u4) to node[above] {$x_3$} (u5);
\draw[-Stealth] (u2) to node {$c_0$\phantom{xx}} (d2); 
\draw[-Stealth] (u3) to node {$c_1$\phantom{xx}} (d3);
\draw[-Stealth] (u4) to node {$c_2$\phantom{xx}} (d4);
\draw[-Stealth] (u5) to node {\phantom{w}$c_3=c$} (d5);
\draw[-Stealth] (u1) to node[near start] {\tiny \phantom{wwi}$d_0$} (d2);
\draw[-Stealth] (u2) to node[near start] {\tiny \phantom{wwi}\tiny $d_1$} (d3);
\draw[-Stealth] (u3) to node[near start] {\tiny \phantom{wwi}\tiny $d_2$} (d4);
\draw[-Stealth] (u4) to node[near start] {\tiny \phantom{wwi}\tiny $d_3$} (d5);
\draw[-Stealth] (d2) to node[below] {$y_1$} (d3);
\draw[-Stealth] (d3) to node[below] {$y_2$} (d4);
\draw[-Stealth] (d4) to node[below] {$y_3$} (d5);
%
%
\draw[color=blue,thick] (1.8,1.5) arc [start angle=0, end angle=180, radius=0.3];
\draw[color=blue,thick] (3.3,1.5) arc [start angle=0, end angle=180, radius=0.3];
\draw[color=blue,thick] (4.8,1.5) arc [start angle=0, end angle=180, radius=0.3];
\draw[color=blue,thick] (1.8,0) arc [start angle=0, end angle=135, radius=0.3];
\draw[color=blue,thick] (3.3,0) arc [start angle=0, end angle=135, radius=0.3];
\draw[color=blue,thick] (4.8,0) arc [start angle=0, end angle=135, radius=0.3];
\draw[densely dotted,color=blue,thick] (2.7,0) arc [start angle=180, end angle=360, radius=0.3];
\draw[densely dotted,color=blue,thick] (4.2,0) arc [start angle=180, end angle=360, radius=0.3];
%
\draw[-Stealth,dashed] (1.55,2.1) to node[above] {$x$} (5.95,2.1);
\draw[-] (1.5,2.05) to (1.5,2.15);
\draw[-] (6,2.05) to (6,2.15);
\end{tikzpicture}\end{center}  }
\caption{A domino diagram showing the calculation of the normal form of $c^{-1}x\,c$ if $\inf(x)=0$. It is a non-obvious fact that the word $y_1\,y_2\,y_3$ obtained by this calculation is in normal form. Also, under the hypothesis that both $x$ and $c^{-1}x\, c$ are rigid, it is true but non-obvious that $d_0=d_3$ and $c_0=c_3=c$, meaning that $y_1\, y_2\,  y_3=c^{-1} x\, c$, as desired.}
\label{F:DominoDiagramInf0}\end{figure}
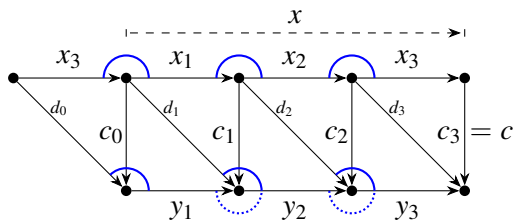

We denote $d_\ell=x_\ell\cdot c$, and recall that this is a simple element by hypothesis. The first calculation step consists in determining the normal form of $x_{\ell-1} d_\ell$ -- we denote the two factors $d_{\ell-1}$ and $y_\ell$, so that $x_{\ell-1}\cdot d_\ell = d_{\ell-1}|y_\ell$. We denote $c_{\ell-1}=x_{\ell-1}^{-1}\cdot d_{\ell-1}$ -- this simple element can be interpreted as ``the initial part of~$d_\ell$ that can be slid into~$x_{\ell-1}$''. 
Then we work our way backwards through the word: $x_{\ell-2}d_{\ell-1}=d_{\ell-2}|y_{\ell-1}$, etc. The last step consists in calculating the normal form $x_\ell\cdot d_1=d_0|y_1$, and $c_0=x_\ell^{-1}d_0$. 
The right domino rule from \cite{DehornoyQuadrNorm} tells us the non-obvious fact that the word $y_1\cdot y_2\cdot\ldots\cdot y_\ell$ thus obtained is in normal form (as indicated by the blue dotted arcs in Figure~\ref{F:DominoDiagramInf0}).

\begin{lemma}
Suppose that $x$ is rigid with $\inf(x)=0$, that $c\preccurlyeq \partial\phi(x)$, and that $c^{-1}\cdot x\cdot c$ is also rigid. Suppose the normal form of $x_\ell|x_1|\ldots|x_\ell\cdot c$ is $d_0|y_1|\ldots |y_\ell$. Then $d_0=d_\ell$, and $c_0=c_\ell=c$. In particular, $c^{-1}xc = y_1|y_2|\ldots|y_\ell$. 
\end{lemma}

\begin{proof}
This result is contained in the statement and proof of Proposition~2.1 of \cite{Gebhardt} -- in that paper, our elements $c_i$ are denoted $u_i$. 
\end{proof}

We have seen how to calculate the normal form of $c^{-1}xc$, provided that both $x$ and $c^{-1} x c$ belong to $SC(x)$, that $c\preccurlyeq \partial\phi(x)$ (i.e.the conjugation represents a gray arrow), and that $\inf(x)=0$. 

If $\inf(x)\neq 0$, then we need a slight modification of the previous rule (see Figure~\ref{F:DominoDiagramInfNeq0} and Example~\ref{E:ConjugationInB4dWithInf}): 
\begin{figure}[htb]
\bigfigure{ 
\begin{center}\begin{tikzpicture}
\node[dot] (u1) at (0,1.5) {};
\node[dot] (u2) at (1.5,1.5) {};
\node[dot] (u2b) at (3,1.5) {};
\node[dot] (u3) at (4.5,1.5) {};
\node[dot] (u4) at (6,1.5) {};
\node[dot] (u5) at (7.5,1.5) {};
\node[dot] (d2) at (1.5,0) {};
\node[dot] (d3) at (3,0) {};
\node[dot] (d3b) at (4.5,0) {};
\node[dot] (d4) at (6,0) {};
\node[dot] (d5) at (7.5,0) {};
\draw[-Stealth] (u1) to node[above] {$x_3$} (u2); 
\draw[-Stealth] (u2) to node[above] {$\Delta^k$} (u2b);
\draw[-Stealth] (u2b) to node[above] {$x_1$} (u3);
\draw[-Stealth] (u3) to node[above] {$x_2$} (u4);
\draw[-Stealth] (u4) to node[above] {$x_3$} (u5);
\draw[-Stealth] (u2) to node {$c_0$\phantom{xx}} (d2); 
\draw[-Stealth] (u3) to node {$c_1$\phantom{xx}} (d3b);
\draw[-Stealth] (u4) to node {$c_2$\phantom{xx}} (d4);
\draw[-Stealth] (u5) to node {\phantom{w}$c_3=c$} (d5);
\draw[-Stealth] (u1) to node[near start] {\tiny \phantom{wwi}$d_0$} (d2);
\draw[-Stealth] (u2) to node[near start] {\tiny \phantom{wwi}\tiny $\tau^{-k}(d_1)$} (d3);
\draw[-Stealth] (u2b) to node[near start] {\tiny \phantom{wwi}\tiny $d_1$} (d3b);
\draw[-Stealth] (u3) to node[near start] {\tiny \phantom{wwi}\tiny $d_2$} (d4);
\draw[-Stealth] (u4) to node[near start] {\tiny \phantom{wwi}\tiny $d_3$} (d5);
\draw[-Stealth] (d2) to node[below] {$\tau^{-k}(y_1)$} (d3);
\draw[-Stealth] (d3) to node[below] {$\Delta^k$} (d3b);
\draw[-Stealth] (d3b) to node[below] {$y_2$} (d4);
\draw[-Stealth] (d4) to node[below] {$y_3$} (d5);
%
%
\draw[color=blue,thick] (1.2,1.5) arc [start angle=180, end angle=90, radius=0.3];
\draw[color=blue,thick] (3.3,1.5) arc [start angle=0, end angle=90, radius=0.3];
\draw[color=blue,thick] (4.8,1.5) arc [start angle=0, end angle=180, radius=0.3];
\draw[color=blue,thick] (6.3,1.5) arc [start angle=0, end angle=180, radius=0.3];
\draw[color=blue,thick] (1.8,0) arc [start angle=0, end angle=135, radius=0.3];
\draw[color=blue,thick] (4.8,0) arc [start angle=0, end angle=135, radius=0.3];
\draw[color=blue,thick] (6.3,0) arc [start angle=0, end angle=135, radius=0.3];
\draw[densely dotted,color=blue,thick] (2.7,0) arc [start angle=180, end angle=270, radius=0.3];
\draw[densely dotted,color=blue,thick] (4.8,0) arc [start angle=360, end angle=270, radius=0.3];
\draw[densely dotted,color=blue,thick] (5.7,0) arc [start angle=180, end angle=360, radius=0.3];
\draw[-Stealth,dashed] (1.55,2.1) to node[above] {$x$} (7.45,2.1);
\draw[-] (1.5,2.05) to (1.5,2.15);
\draw[-] (7.5,2.05) to (7.5,2.15);
\end{tikzpicture}\end{center} 
}
\caption{A domino diagram showing the calculation of the normal form of $c^{-1}x\,c$, in the case $\inf(x)\neq 0$. Again, the word $\Delta^k y_1\,y_2\,y_3$ obtained by this calculation is in normal form. Also, under the hypothesis that both $x$ and $c^{-1}x\, c$ are rigid, $d_0=d_3$ and $c_0=c_3=c$, meaning that $\Delta^k y_1\, y_2\,  y_3=c^{-1} x\, c$.}
\label{F:DominoDiagramInfNeq0}\end{figure}
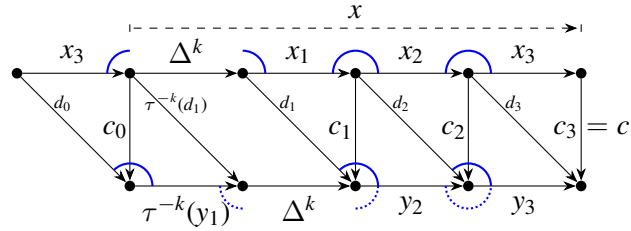
in the very last step, we do \emph{not} determine the letters $d_0$ and $y_1$ by calculating the normal form $d_0|y_1$ of $x_\ell\cdot d_1$; instead, we determine two letters, which we call $d_0$ and $\tau^{-k}(y_1)$, by calculating the normal form $d_0|\tau^{-k}(y_1)$ of $x_\ell\cdot \tau^{-k}(d_1)$. As previously, we let $c_0 = x_\ell^{-1} d_0$. By the same arguments as above, we have $d_0=d_\ell$ and $c_0=c_\ell=c$. Thus, $c^{-1}xc=\tau^{-k}(y_1)\,\Delta^k\, y_2\ldots y_\ell$. The latter word is almost in normal form: in order to obtain its normal form, it suffices to slide $\Delta^k$ to the start, using the rule $\tau^{-k}(y_1)\Delta^k=\Delta^k\, y_1$. Finally, the normal form of $c^{-1} x c$ is $\Delta^k\,y_1|\ldots|y_\ell$.


\subsection{The structure of the conjugacy graph}

\begin{lemma}\label{L:PowersAndConjGraphs}
Let $G$ be a Garside group. Suppose $x\in G$ is rigid, and $n,N$ are two integers with $n|N$. Denoting $d=\frac{N}{n}$, we have:
\begin{enumerate}
\item The function $\pi^d\co SC(x^n)\longrightarrow SC(x^N)$, $y\mapsto y^d$ is an injection. In particular, $|SC(x^n)|\leqslant |SC(x^N)|$.
\item The injection $\pi^d$ sends each orbit under cycling and~$\tau$ of $SC(x^n)$ bijectively to an orbit under cycling and~$\tau$ of $SC(x^N)$.
\item The function $\pi^d$ gives rise to an inclusion of the conjugacy graph of~$x^n$ in the conjugacy graph of~$x^N$.
\item Suppose that the natural injection of centralizers $C(x^n) \hookrightarrow C(x^N)$ is also surjective. Then the image of the inclusion mentioned in~(3) of the conjugacy graph of~$x^n$ in the conjugacy graph of~$x^N$ is an induced subgraph.
\end{enumerate} 
\end{lemma}

\begin{proof} 
(1) If $y$ is a rigid element conjugate to $x^n$, then $y^d$ is a rigid conjugate of $x^N$, so $\pi^d$ is well-defined. Next, we claim that the normal form of~$y$ can be reconstructed from the normal form of~$y^d$ (which implies injectivity). Indeed, since $y$ is rigid, $\inf(y)=\frac{\inf(y^d)}{d}$, and the non-$\Delta$ factors of the normal form of~$y$ are the same as the last $\frac{\ell_{can}(y^d)}{d}$ factors of the normal form of~$y^d$.

(2) is a consequence of the fact that $\pi^d$~commutes with cycling and with~$\tau$. 

(3) Before proving~(3), we warn the reader that minimality of arrows may not be preserved by~$\pi^d$, i.e.it can happen that a minimal edge in the conjugacy graph of~$x^n$ is split into two minimal edges by~$\pi^d$.

From point (2) we know that the function~$\pi^d$ induces a well-defined map from the vertices of the conjugacy graph of~$x^n$ to the conjugacy graph of~$x^N$. Let us now think about the edges.

Suppose in the conjugacy graph of~$x^n$ we have a black edge from a vertex represented by an element $y$ of~$SC(x^n)$ to a vertex represented by $z=c^{-1}yc\in SC(x^n)$, where $c\preccurlyeq \iota(y)$. Then $c^{-1}y^dc=z^d$, and also $c\preccurlyeq \iota(y^d)$, because (by rigidity of~$y$) we have $\iota(y^d)=\iota(y)$. This means that in the conjugacy graph of~$x^N$, there is still a black edge between the vertices represented by $y^d$ and~$z^d$. The same argument works for gray edges. In summary, $\pi^d$ induces an inclusion of the conjugacy graph of $x^n$ in the conjugacy graph of~$x^N$.

(4) Suppose that $y,z\in SC(x^n)$ represent distinct vertices of the conjugacy graph of~$x^n$, and that in the conjugacy graph of~$x^N$ there is a black edge connecting the vertices represented by $y^d$ and~$z^d$. Let $c$ be the conjugating element: $c\preccurlyeq \iota(y^d)$. We have to prove that there is also a black arrow in the conjugacy graph of $x^n$ from $y$ to~$z$ with conjugating element~$c$. 
For that, we observe that $c\preccurlyeq \iota(y)=\iota(y^d)$, by the rigidity of~$y$; moreover, since $y$ and~$z$ are both contained in $SC(x^n)$, there exists an element $\tilde c$ which conjugates: $\tilde c^{-1} y \tilde c = z$. By taking $d$th powers, we obtain $\tilde c^{-1} y^d \tilde c = z^d$. 
Together with the fact that $c^{-1} y^d c = z^d$ we obtain $\tilde c c^{-1} y^d c \tilde c^{-1} = y^d$. Since the centralizers of $y^d$ and of~$y$ coincide, we deduce that $\tilde c c^{-1} y c \tilde 
c^{-1} = y$, and conclude that $c^{-1} y c = \tilde c^{-1} y \tilde c = z$.
\end{proof}

\begin{remark}
Lemma~\ref{L:PowersAndConjGraphs}(4) will not actually be used in the rest of the paper. It is, however, quite a powerful statement: for instance, in the context of braid groups, the hypothesis $C(x^n)=C(x^N)$ is verified for all pseudo-Anosov braids. This is an immediate consequence of the fact (proved in \cite{GMRootUnique} that pseudo-Anosov braids have unique roots: if $y^{-1} x^l y = x^l$, then $(y^{-1}xy)^l=x^l$, and by the uniqueness of roots $y^{-1}x^ky=(y^{-1}xy)^k=x^k$. (The fact that $C(x^k)=C(x^l)$ for a pseudo-Anosov braid~$x$ can also be deduced from the results of~\cite{GMWiCentralizer}.)
\end{remark}


\subsection{Sliding circuits set and super summit set}

We recall that for any element~$x$ of~$G$ there is another well-known characteristic subset of the conjugacy class of~$x$, called the \emph{super summit set} $SSS(x)$ which satisfies $SC(x)\subseteq SSS(x)$. We recall the definition from \cite{ElRifaiMorton}: an element $y$ of~$G$ belongs to $SSS(x)$ if it is conjugate to~$x$, if $\inf(y)$ is as large as possible among conjugates of~$x$, \emph{and} if $\sup(y)$ is as small as possible among conjugates of~$x$. It is a non-obvious fact that this subset is always non-empty, and that it coincides with the set of conjugates $y$ of~$x$ whose canonical length~$\ell_{can}(y)$ is as small as possible among conjugates of~$x$. 

\begin{proposition}\label{P:PowerOfSSSNotRigid}
Suppose that $x$ is a rigid element of~$G$, but that its conjugate $y\in SSS(x)$ is not rigid. Then no positive power $y^n$ is rigid.
\end{proposition}

In other words, if $x$ is rigid and $y\in SSS(x)\setminus SC(x)$, then no power $y^n$ belongs to $SC(x^n)$. Yet another way of saying this is:
for elements possessing a rigid power \emph{and} a rigid conjugate, the property of being in its own $SSS$ is actually equivalent to the (a priori much stronger) condition of being in its own~$SC$. 

\begin{proof}
For every positive integer~$n$, the element $x^n$ is rigid. Therefore $\ell_{can}(x^n)=n\cdot \ell_{can}(x)$, and no conjugate of $x^n$ has a smaller canonical length than that. In particular, $\ell_{can}(y^n)\geqslant n\cdot \ell_{can}(x)$.
On the other hand, $y\in SSS(y)=SSS(x)$, so $\ell_{can}(y)=\ell_{can}(x)$. This implies that $\ell_{can}(y^n)\leqslant n\cdot \ell_{can}(y)=n\cdot \ell_{can}(x)$. We have proven that $\ell_{can}(y^n)=n\cdot \ell_{can}(x)$ for every positive integer~$n$. Since this is the smallest canonical length in the conjugacy class of $y^n$, it also follows that $\inf(y^n)=n\inf(y)=n\inf(x)$ for every integer $n$.

This implies that the sequence of initial factors $\iota(y), \iota(y^2),\ldots$ is increasing, in the sense that each is a prefix of the next: $\iota(y)\preccurlyeq \iota(y^2)\preccurlyeq \iota(y^3) \preccurlyeq\ldots$. Indeed, if $p=\inf(x)$ and $r=\ell_{can}(x)$, the left normal form of $y^n$, for any given $n>0$, is as follows: $\Delta^{pn}s_1|\cdots|s_{rn}$. When we multiply it by another copy of $y$, we obtain 
$$
y^{n+1}=\Delta^{pn}s_1\cdots s_{rn} \Delta^p y_1\cdots y_r = \Delta^{p(n+1)} \tau^p(s_1)\cdots \tau^p(s_{rn})y_1\cdots y_r.
$$
Since the canonical length of $y^{n+1}$ equals $r(n+1)$, it follows that no factor $\Delta$ is created when computing the normal form of $\tau^p(s_1)\cdots \tau^p(s_{rn})y_1\cdots y_r$. So the first factor in the normal form of this element is $\tau^p(s_1)t$ for some simple element $t$. Therefore, the initial factor of $y^{n+1}$ is 
$$
\iota(y^{n+1})=\tau^{-p(n+1)}(\tau^p(s_1)t)= \tau^{-pn}(s_1)\tau^{-p(n+1)}(t) = \iota(y^n)\tau^{-p(n+1)}(t).
$$
Since $\tau^{-p(n+1)}(t)$ is simple, we obtain $\iota(y^n)\preccurlyeq \iota(y^{n+1})$, as claimed.
(This argument follows the proof of Lemma 3.28 in~\cite{Birman-Gebhardt-GM1}).

In particular, $\iota(y)\preccurlyeq \iota(y^n)$. By a similar argument, $\phi(y)\succcurlyeq \phi(y^n)$.

Now by hypothesis, $y$~is not rigid, i.e.the product $\phi(y)\cdot \iota(y)$ is not left-weighted; this means that $\iota(y)$ has a non-empty prefix~$i$ such that $\phi(y)\cdot i$ is simple. \emph{A fortiori}, the product $\phi(y^n)\cdot \iota(y^n)$ is not left-weighted, either (as the same element $i$ is still a prefix of $\iota(y^n)$ and $\phi(y^n)\cdot i$ is still simple). This means that $y^n$ is not rigid.
\end{proof}

\begin{example} We return to Example~\ref{E:SimpleExInB4classical}. There, the rigid braid $x=21|12|2132$, with $\inf(x)=0$ and $\sup(x)=3$, has a conjugate $y=1^{-1}x1=\Delta^{-1} 1 2 1 3 2 | 2 1 3 2 1 | 2 3 | 3 2$ which is not rigid but whose square is rigid. As predicted by Lemma~\ref{P:PowerOfSSSNotRigid}, $y$ does not even belong to $SSS(x)$, as witnessed by the fact that $\inf(y)=-1$.
\end{example}


\section{More examples}\label{S:Examples}

In this section we present some examples which we find enlightening. All calculations were performed with the computer programs Cbraid \cite{GMcbraid} (for~$\calB_m$) and GAP3 \cite{GAP3, Chevie} (for~$\calB_m^*$). Whenever we say that the sequence $|SC(x^n)|_{n\in\mathbb{N}}$ ``appears to be'' periodic of some period, we mean that this is what our (necessarily finite) calculations indicate.

\begin{example}
In~$\calB_5$, if we set $x=213243|34|432$, then the sequence $(|SC(x^n)|)_{n\in\mathbb N}$ appears to be periodic of period~3, with $|SC(x)|=|SC(x^2)|=6$ and $|SC(x^3)|=42$.
The conjugacy graph of~$x^3$ consists of three vertices which are connected by three gray arrows in a cyclic manner.)
\end{example}

\begin{example}
In~$\calB_6$, if we set $x=243215432|24$, then the sequence $(|SC(x^n)|)_{n\in\mathbb N}$ appears to be periodic of period~6, and $|SC(x)|=4$, $|SC(x^2)|=12$, $|SC(x^3)|=28$, $|SC(x^4)|=12$, $|SC(x^5)|=4$ and $|SC(x^6)|=84$.  
\end{example}

\begin{example}\label{E:TheBigOne}
The following example might help for understanding the general case. We consider the 8-strand braid group, equipped with its classical Garside structure, and the element
$$
x = 246|24654321765432 \in \calB_8
$$
Calculations with the computer program \cite{GMcbraid} indicate that this element~$x$ is rigid and pseudo-Anosov, with $\inf(x)=0$ and $\sup(x)=2$. Moreover, the sequence $(|SC(x^n)|)_{n\in\mathbb N}$ appears to be periodic of period~12, with
\begin{center}
\begin{tabular}{r||c|c|c|c|c|c|c|c|c|c|c|c|}
$n$ & 1 & 2 & 3 & 4 & 5 & 6 & 7 & 8 & 9 & 10 & 11 & 12 \\
\hline
\phantom{$\frac{1^1}{1}$}$|SC(x^n)|$ & 4 & 12 & 40 & 76 & 4 & 120 & 4 & 76 & 40 & 12 & 4 & 760 \\
\end{tabular}
\end{center}

One possible interpretation is that the 8-strand braid~$x$ has some kind of symmetry which only reveals itself in the twelfth power~$x^{12}$. Making this intuition more precise might be a great step towards proving Conjecture~\ref{C:MainConjecture}.

Figure~\ref{F:TheBigOne} shows the conjugacy graph of $x^{12}$. This graph nicely illustrates the fact that the possible values of $r(y)$ among conjugates $y$ of~$x$ are 1,2,3,4,6 and 12. The red dot is present in the conjugacy graph of $x^n$ for every~$n$, where it is represented by the braid~$x^n$. The orange dot is present in the conjugacy graph of $x^n$ if $n$ is even: it is represented by $y^2$ for some non-rigid conjugate $y$ of~$x$ with $r(y)=2$ -- we say the element $y^2$ is at \emph{level 2}. The yellow dots are in the graph of $x^n$ if $3|n$ (they are represented by level 3 elements), the green dots if $4|n$, the blue dots if $6|n$, and the white dots if $12|n$.

\begin{figure}[htb]
\bigfigure{
\definecolor{myblue}{rgb}{0,0.2,1}
\begin{center}\begin{tikzpicture}
\node[bigdot,label={right:Level \ 1},fill=red] (101) at (-7.7,3.1) {}; 
\node[bigdot,label={right:Level \ 2},fill=orange] (102) at (-7.7,3.6) {}; %
\node[bigdot,label={right:Level \ 3},fill=yellow] (103) at (-7.7,4.1) {}; %
\node[bigdot,label={right:Level \ 4},fill=green] (104) at (-7.7,4.6) {}; %
\node[bigdot,label={right:Level \ 6},fill=myblue] (106) at (-7.7,5.1) {}; %
\node[bigdot,label={right:Level 12},fill=white] (112) at (-7.7,5.6) {}; %
\node[bigdot,white] (99) at (7.5,-5.6) {}; 
\node[bigdot,label={\verlab{1}},fill=red] (1) at (3.5,0) {}; %
\node[bigdot,label={\verlab{11}},fill=orange] (11) at (-3.5,0) {}; %
\node[bigdot,label={\verlab{23}},fill=yellow] (23) at (0.5,0) {}; %
\node[bigdot,label={\verlab{13}},fill=yellow] (13) at (1.5,0) {}; %
\node[bigdot,label={\verlab{5}},fill=yellow] (5) at (2.5,0) {}; %
\node[bigdot,label={\verlab{9}},fill=green] (9) at (7.5,0) {}; %
\node[bigdot,label={\verlab{21}},fill=green] (21) at (0,-5.6) {}; %
\node[bigdot,label={\verlab{33}},fill=green] (33) at (-7.5,0) {}; %
\node[bigdot,label={\verlab{3}},fill=green] (3) at (0,5.6) {}; %
\node[bigdot,label={\verlab{19}},fill=myblue] (19) at (-2.5,0) {}; %
\node[bigdot,label={\verlab{31}},fill=myblue] (31) at (-1.5,0) {}; %
\node[bigdot,label={\verlab{41}},fill=myblue] (41) at (-0.5,0) {}; %
\node[bigdot,label={\verlab{7}},fill=white] (7) at (0,4.6) {}; %
\node[bigdot,label={\verlab{17}},fill=white] (17) at (0,3.6) {}; %
\node[bigdot,label={\verlab{29}},fill=white] (29) at (0,2.6) {}; %
\node[bigdot,label={\verlab{15}},fill=white] (15) at (6.5,0) {}; %
\node[bigdot,label={\verlab{25}},fill=white] (25) at (5.5,0) {}; %
\node[bigdot,label={\verlab{37}},fill=white] (37) at (4.5,0) {}; %
\node[bigdot,label={\verlab{27}},fill=white] (27) at (0,-4.6) {}; %
\node[bigdot,label={\verlab{35}},fill=white] (35) at (0,-3.6) {}; %
\node[bigdot,label={\verlab{45}},fill=white] (45) at (0,-2.6) {}; %
\node[bigdot,label={\verlab{47}},fill=white] (47) at (-4.5,0) {}; %
\node[bigdot,label={\verlab{43}},fill=white] (43) at (-5.5,0) {}; %
\node[bigdot,label={\verlab{39}},fill=white] (39) at (-6.5,0) {}; %
\draw[-Stealth,thick,black] (1) to node[near start] {\tiny{$\times\!3$}\phantom{.}} (5);
\draw[-Stealth,thick,gray] (1) to[bend right=10] node[near start] {\tiny{$\times 4$}} (3);
\draw[-Stealth,thick,black] (13) to[bend left=40] node[near end] {\tiny{\arrlab{$\times 1$}}} (1);
\draw[-Stealth,thick,gray] (21)[bend right=10] to[near end] node {\tiny{\arrlab{$\times 1$}}} (1);
\draw[-Stealth,thick,black] (11) to node[near start] {\tiny{$\times\!3$}} (19);
\draw[-Stealth,thick,gray] (11) to[bend right=10] node[near start] {\tiny{$\times 2$}} (21);
\draw[-Stealth,thick,gray] (3) to[bend right=10] node[near end] {\tiny{\arrlab{$\times 1$}}} (11);
\draw[-Stealth,thick,black] (31) to[bend left=40] node[near end] {\tiny{\arrlab{$\times 1$}}} (11);
\draw[-Stealth,thick,gray] (9)[bend left=35] to[near start] node {\tiny{$\times 2$}} (21);
\draw[-Stealth,thick,gray] (21) to[bend left=35] node[near end] {\tiny{\arrlab{$\times 1$}}} (33);
\draw[-Stealth,thick,gray] (33) to[bend left=35] node[near start] {\tiny{\arrlab{$\times 1$}}} (3);
\draw[-Stealth,thick,gray] (3) to[bend left=35] node[near end] {\tiny{$\times 2$}} (9);
\draw[-Stealth,thick,black] (17) to[bend left=40] node[near end] {\tiny{\arrlab{$\times 1$}}} (3);
\draw[-Stealth,thick,black] (3) to node[near start] {\tiny{$\stackrel{\phantom{x}}{\times\!3}$}} (7);
\draw[-Stealth,thick,black] (9) to node[near start] {\tiny{$\times\!3$\phantom{.}}} (15);
\draw[-Stealth,thick,black] (25) to[bend right=40] node[near end] {\tiny{\arrlab{$\times 1$}}} (9);
\draw[-Stealth,thick,black] (35) to[bend left=40] node[near end] {\tiny{\arrlab{$\times 1$}}} (21);
\draw[-Stealth,thick,black] (21) to node[near start] {\tiny{$\stackrel{\times\!3}{\phantom{x}}$}} (27);
\draw[-Stealth,thick,black] (43) to[bend right=40] node[near end] {\tiny{\arrlab{$\times 1$}}} (33);
\draw[-Stealth,thick,black] (33) to node[near start] {\tiny{$\times\!3$}} (39);
\draw[-Stealth,thick,gray] (7) to[bend right=8] node[near end] {\tiny{\arrlab{$\times 1$}}} (19);
\draw[-Stealth,thick,gray] (19) to[bend right=8] node[near start] {\tiny{$\times 2$}} (27);
\draw[-Stealth,thick,gray] (27) to[bend right=8] node[near end] {\tiny{\arrlab{$\times 1$}}} (5);
\draw[-Stealth,thick,gray] (5) to[bend right=8] node[near start] {\tiny{$\times 4$}} (7);
\draw[-Stealth,thick,gray] (17) to[bend right=6] node[near end] {\tiny{\arrlab{$\times 1$}}} (31);
\draw[-Stealth,thick,gray] (31) to[bend right=6] node[near start] {\tiny{$\times 2$}} (35);
\draw[-Stealth,thick,gray] (35) to[bend right=6] node[near end] {\tiny{\arrlab{$\times 1$}}} (13);
\draw[-Stealth,thick,gray] (13) to[bend right=6] node[near start] {\tiny{$\times 4$}} (17);
\draw[-Stealth,thick,gray] (29) to[bend right=4] node[near end] {\tiny{\arrlab{$\times 1$}}} (41);
\draw[-Stealth,thick,gray] (41) to[bend right=4] node[near start] {\tiny{$\times 2$}} (45);
\draw[-Stealth,thick,gray] (45) to[bend right=4] node[near end] {\tiny{\arrlab{$\times 1$}}} (23);
\draw[-Stealth,thick,gray] (23) to[bend right=4] node[near start] {\tiny{$\times 4$}} (29);
\draw[-Stealth,thick,black] (5) to node[near start] {\tiny{$\times\!2$\phantom{.}}} (13);
\draw[-Stealth,thick,black] (13) to node[near start] {\tiny{\arrlab{$\times\!1$}}} (23);
\draw[-Stealth,thick,black] (23) to[bend left=40] node[near start] {\tiny{\arrlab{$\times 1$}}} (5);
\draw[-Stealth,thick,black] (19) to node[near start] {\tiny{$\times\!2$}} (31);
\draw[-Stealth,thick,black] (31) to node[near start] {\tiny{\arrlab{$\times\!1$}}} (41);
\draw[-Stealth,thick,black] (41) to[bend left=40] node[near start] {\tiny{\arrlab{$\times 1$}}} (19);
\draw[-Stealth,thick,gray] (7) to[bend left=33] node[near end] {\tiny{$\times 2$}} (15);
\draw[-Stealth,thick,gray] (15) to[bend left=33] node[near start] {\tiny{$\times 2$}} (27);
\draw[-Stealth,thick,gray] (27) to[bend left=33] node[near end] {\tiny{\arrlab{$\times 1$}}} (39);
\draw[-Stealth,thick,gray] (39) to[bend left=33] node[near start] {\tiny{\arrlab{$\times 1$}}} (7);
\draw[-Stealth,thick,gray] (17) to[bend left=31] node[near end] {\tiny{$\times 2$}} (25);
\draw[-Stealth,thick,gray] (25) to[bend left=31] node[near start] {\tiny{$\times 2$}} (35);
\draw[-Stealth,thick,gray] (35) to[bend left=31] node[near end] {\tiny{\arrlab{$\times 1$}}} (43);
\draw[-Stealth,thick,gray] (43) to[bend left=31] node[near start] {\tiny{\arrlab{$\times 1$}}} (17);
\draw[-Stealth,thick,gray] (29) to[bend left=29] node[near end] {\tiny{$\times 2$}} (37);
\draw[-Stealth,thick,gray] (37) to[bend left=29] node[near start] {\tiny{$\times 2$}} (45);
\draw[-Stealth,thick,gray] (45) to[bend left=29] node[near end] {\tiny{\arrlab{$\times 1$}}} (47);
\draw[-Stealth,thick,gray] (47) to[bend left=29] node[near start] {\tiny{\arrlab{$\times 1$}}} (29);
\draw[-Stealth,thick,black] (7) to node[near start] {\tiny{$\stackrel{\phantom{X}}{\times\!2}$}} (17);
\draw[-Stealth,thick,black] (17) to node[near start] {\tiny{\arrlab{$\times\!1$}}} (29);
\draw[-Stealth,thick,black] (29)[bend left=40] to node[near start] {\tiny{\arrlab{$\times 1$}}} (7);
\draw[-Stealth,thick,black] (15) to node[near start] {\tiny{$\times\!2$\phantom{.}}} (25);
\draw[-Stealth,thick,black] (25) to node[near start] {\tiny{\arrlab{$\times\!1$}}} (37);
\draw[-Stealth,thick,black] (37) to[bend right=40] node[near start] {\tiny{\arrlab{$\times 1$}}} (15);
\draw[-Stealth,thick,black] (27) to node[near start] {\tiny{$\stackrel{\times\!2}{{\phantom{x}}}$}} (35);
\draw[-Stealth,thick,black] (35) to node[near start] {\tiny{\arrlab{$\times\!1$}}} (45);
\draw[-Stealth,thick,black] (45) to[bend left=40] node[near start] {\tiny{\arrlab{$\times 1$}}} (27);
\draw[-Stealth,thick,black] (39) to node[near start] {\tiny{$\times\!2$}} (43);
\draw[-Stealth,thick,black] (43) to node[near start] {\tiny{\arrlab{$\times\!1$}}} (47);
\draw[-Stealth,thick,black] (47) to[bend right=40] node[near start] {\tiny{\arrlab{$\times 1$}}} (39);
\end{tikzpicture}\end{center} 
}
\caption{The conjugacy graph of $x^{12}$. However, for the sake of clarity, not \emph{all} black and gray arrows are shown but only the \emph{minimal} ones, i.e.~only those arrows that cannot be obtained as the composition of two or more arrows of the same color.} 
\label{F:TheBigOne}\end{figure}
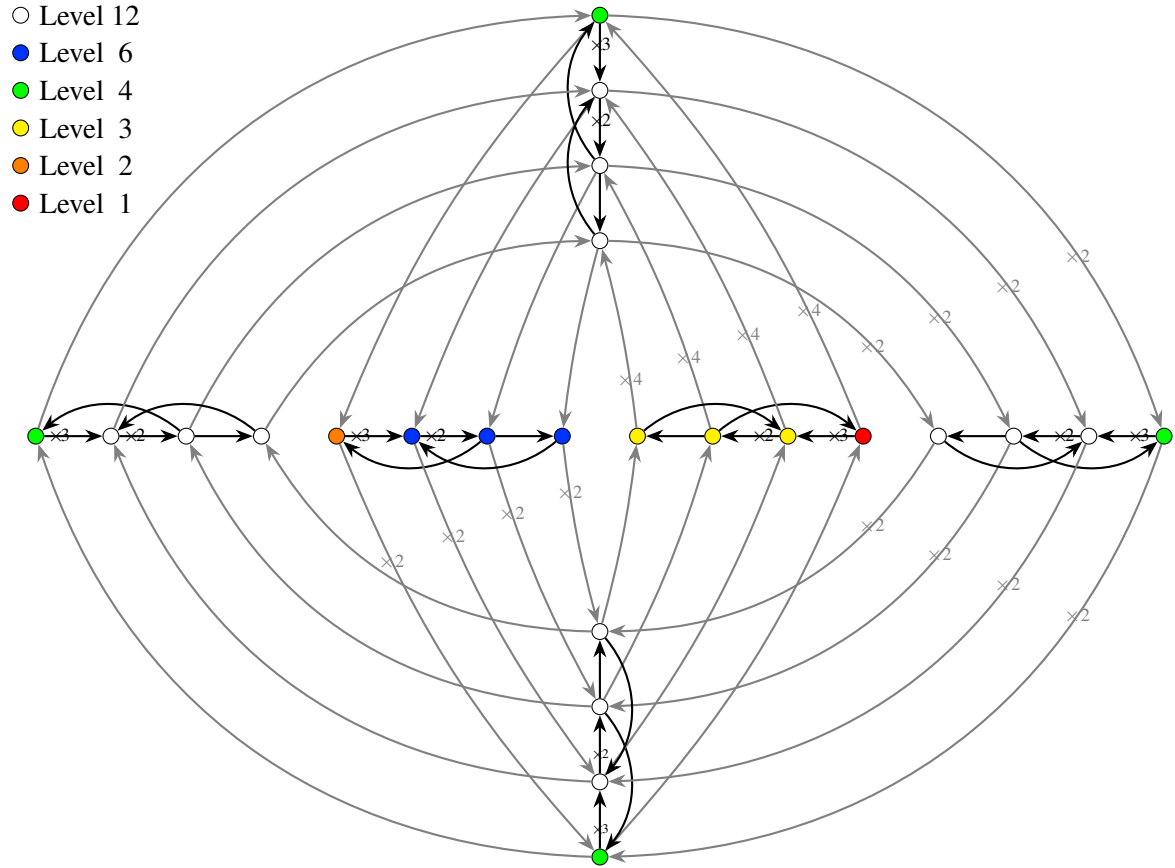

Here is an example of an element~$y$ which is conjugate to~$x$, which is not rigid, but whose twelfth power is: we take $y$ to be $x$, conjugated by the braid $c=16$
$$
y = -6 \, -\hspace{-2pt}1 \, \cdot \, x \, \cdot \, 1 \, 6 =
\Delta^{-1} (\Delta \sigma_1^{-1})|246543217654321|46|6
$$
where $\Delta \sigma_1^{-1}$ is a simple element (that we did not write explicitly for the sake of clarity). 

\begin{figure}
\bigfigure{ 
\begin{tikzpicture}[scale=0.14, every node/.style={scale=0.14}]
\pic[
rotate=90,
braid/.cd,
gap=0.2, 
] {braid={s_2-s_4-s_6 s_2-s_4-s_6 s_5-s_7 s_4-s_6 s_3-s_5 s_2-s_4 s_1-s_3 s_2   s_2-s_4-s_6 s_2-s_4-s_6 s_5-s_7 s_4-s_6 s_3-s_5 s_2-s_4 s_1-s_3 s_2   s_2-s_4-s_6 s_2-s_4-s_6 s_5-s_7 s_4-s_6 s_3-s_5 s_2-s_4 s_1-s_3 s_2   s_2-s_4-s_6 s_2-s_4-s_6 s_5-s_7 s_4-s_6 s_3-s_5 s_2-s_4 s_1-s_3 s_2   s_2-s_4-s_6 s_2-s_4-s_6 s_5-s_7 s_4-s_6 s_3-s_5 s_2-s_4 s_1-s_3 s_2   s_2-s_4-s_6 s_2-s_4-s_6 s_5-s_7 s_4-s_6 s_3-s_5 s_2-s_4 s_1-s_3 s_2   s_2-s_4-s_6 s_2-s_4-s_6 s_5-s_7 s_4-s_6 s_3-s_5 s_2-s_4 s_1-s_3 s_2   s_2-s_4-s_6 s_2-s_4-s_6 s_5-s_7 s_4-s_6 s_3-s_5 s_2-s_4 s_1-s_3 s_2   s_2-s_4-s_6 s_2-s_4-s_6 s_5-s_7 s_4-s_6 s_3-s_5 s_2-s_4 s_1-s_3 s_2   s_2-s_4-s_6 s_2-s_4-s_6 s_5-s_7 s_4-s_6 s_3-s_5 s_2-s_4 s_1-s_3 s_2   s_2-s_4-s_6 s_2-s_4-s_6 s_5-s_7 s_4-s_6 s_3-s_5 s_2-s_4 s_1-s_3 s_2   s_2-s_4-s_6 s_2-s_4-s_6 s_5-s_7 s_4-s_6 s_3-s_5 s_2-s_4 s_1-s_3 s_2}};
\draw (0,-1) -- (0,8);
\draw (1.25,-1) -- (1.25,8);
\draw (8.25,-1) -- (8.25,8);
\draw (9.25,-1) -- (9.25,8);
\draw (16.25,-1) -- (16.25,8);
\draw (17.25,-1) -- (17.25,8);
\draw (24.25,-1) -- (24.25,8);
\draw (25.25,-1) -- (25.25,8);
\draw (32.25,-1) -- (32.25,8);
\draw (33.25,-1) -- (33.25,8);
\draw (40.25,-1) -- (40.25,8);
\draw (41.25,-1) -- (41.25,8);
\draw (48.25,-1) -- (48.25,8);
\draw (49.25,-1) -- (49.25,8);
\draw (56.25,-1) -- (56.25,8);
\draw (57.25,-1) -- (57.25,8);
\draw (64.25,-1) -- (64.25,8);
\draw (65.25,-1) -- (65.25,8);
\draw (72.25,-1) -- (72.25,8);
\draw (73.25,-1) -- (73.25,8);
\draw (80.25,-1) -- (80.25,8);
\draw (81.25,-1) -- (81.25,8);
\draw (88.25,-1) -- (88.25,8);
\draw (89.25,-1) -- (89.25,8);
\draw (96.4,-1) -- (96.4,8);
\end{tikzpicture} 
\bigskip

\begin{tikzpicture}[scale=0.14, every node/.style={scale=0.14}]
\pic[
rotate=90,
braid/.cd,
gap=0.2, 
] {braid={s_2 s_1-s_3 s_4 s_3-s_5 s_6 s_5-s_7 s_4-s_6 s_3-s_5 s_2-s_4 s_1-s_3 s_2   s_2 s_1-s_3 s_2-s_4 s_1-s_3-s_5 s_4-s_6 s_3-s_5-s_7   s_1-s_3-s_7   s_1-s_3 s_2-s_4-s_7 s_1-s_3 s_2   s_2 s_1-s_3 s_2-s_4 s_1-s_3-s_5-s_7   s_1 s_2-s_5-s_7   s_2 s_1-s_3-s_5-s_7 s_6 s_5 s_4 s_3 s_2   s_2 s_1-s_3 s_4 s_3-s_5 s_6 s_5-s_7   s_1-s_3-s_5   s_1-s_3-s_5 s_6-s_4 s_5-s_3 s_4-s_2 s_3-s_1 s_2   s_2 s_1-s_3 s_2-s_4 s_1-s_3-s_5 s_4-s_6 s_3-s_5-s_7   s_1-s_3 s_2-s_4 s_1-s_3-s_7 s_2   s_2 s_1-s_3 s_2-s_4 s_1-s_3-s_5-s_7 s_2   s_2 s_1-s_3-s_5-s_7   s_3-s_5-s_7   s_3-s_5-s_7 s_6 s_5 s_4 s_3 s_2   s_2 s_1-s_3 s_4 s_3-s_5 s_6 s_5-s_7   s_1-s_3-s_5 s_4-s_6 s_3-s_5 s_2-s_4 s_1-s_3 s_2   s_2 s_1-s_3 s_2-s_4 s_1-s_3-s_5 s_4-s_6 s_3-s_5-s_7 s_2-s_4 s_1-s_3 s_2   s_2 s_1-s_3 s_2-s_4 s_1-s_3-s_5-s_7   s_1-s_5-s_7   s_1 s_2-s_5-s_7   s_2 s_1-s_3-s_5-s_7   s_3-s_5-s_7 s_6 s_5 s_4 s_3 s_2}};
\draw (0,-1) -- (0,8);
\draw (11.25,-1) -- (11.25,8);
\draw (17.25,-1) -- (17.25,8);
\draw (18.25,-1) -- (18.25,8);
\draw (22.25,-1) -- (22.25,8);
\draw (26.25,-1) -- (26.25,8);
\draw (28.25,-1) -- (28.25,8);
\draw (35.25,-1) -- (35.25,8);
\draw (41.25,-1) -- (41.25,8);
\draw (42.25,-1) -- (42.25,8);
\draw (48.25,-1) -- (48.25,8);
\draw (54.25,-1) -- (54.25,8);
\draw (58.25,-1) -- (58.25,8);
\draw (63.25,-1) -- (63.25,8);
\draw (65.25,-1) -- (65.25,8);
\draw (66.25,-1) -- (66.25,8);
\draw (72.25,-1) -- (72.25,8);
\draw (78.25,-1) -- (78.25,8);
\draw (84.25,-1) -- (84.25,8);
\draw (93.25,-1) -- (93.25,8);
\draw (97.25,-1) -- (97.25,8);
\draw (98.25,-1) -- (98.25,8);
\draw (100.25,-1) -- (100.25,8);
\draw (102.25,-1) -- (102.25,8);
\draw (108.5,-1) -- (108.5,8);
\end{tikzpicture} 
}
\caption{The braids $x^{12}$ (above) and~$y$, its conjugate by $\sigma_1 \sigma_6$ (below). Thus the upper picture shows a representative of the red dot in Figure~\ref{F:TheBigOne}, and the lower one shows a representative of one of the white dots.}
\label{F:x12andconjugate}
\end{figure}
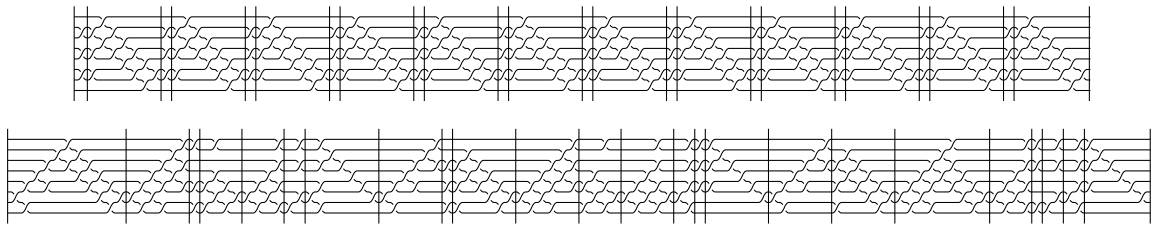

As expected from Lemma~\ref{P:PowerOfSSSNotRigid}, $y$ does not belong to its super summit set -- indeed, we have $\inf(y) = -1 = \inf(x)-1$ and $\sup(y) = 3 = \sup(x)+1$. 
Here is the list of infima and suprema of powers of $x$ and $y$:
\begin{center}
\begin{tabular}{r||c|c|c|c|c|c|c|c|c|c|c|c|}
$n$ & 1 & 2 & 3 & 4 & 5 & 6 & 7 & 8 & 9 & 10 & 11 & 12 \\
\hline\hline
$\phantom{{^{I^I}}}\inf(x^n)$ & 0 & 0 & 0 & 0 & 0 & 0 & 0 & 0 & 0 & 0 & 0 & 0 \\
\hline
$\phantom{{^{I^I}}}\inf(y^n)$ & -1 & -1 & -1 & 0 & -1 & -1 & -1 & 0 & -1 & -1 & -1 & 0 \\
\hline
$\phantom{{^{I^I}}}\sup(x^n)$ & 2 & 4 & 6 & 8 & 10 & 12 & 14 & 16 & 18 & 20 & 22 & 24 \\
\hline
$\phantom{{^{I^I}}}\sup(y^n)$ & 3 & 5 & 6 & 9 & 11 & 12 & 15 & 17 & 18 & 21 & 23 & 24 \\
\end{tabular}
\end{center}
\end{example}

\begin{remark}
We know already (see Remark~\ref{R:DiffDefnsConjGraph}) that in the conjugacy graph of $x^n$, for any rigid element~$x$ and any integer~$n$, we can get from any level~1 vertex to any maximal-level vertex by a path along gray arrows, followed by a path along black arrows (or vice versa). We conjecture that it is actually possible to get from \emph{some} level~1 vertex to \emph{some} maximal-level vertex by a path along a single gray arrow and a single black arrow (these arrows may not be minimal). We will see in Lemma~\ref{L:LimitOnPower} that this conjecture would imply an upper bound on the height of the maximal level, and thus essentially prove our main conjecture.
\end{remark}

\begin{example}
Here is a generalization of Example~\ref{E:TheBigOne}. In~$\calB_{2m}$, the element 
$$x=(2 4  6 8 \ldots 2m-4 2m-2 )^2 2m-3 2m-4 \ldots 3 21 2m-1 2m-2 2m-3 \ldots 4 32$$ 
is a rigid pseudo-Anosov braid, and it appears that the sequence $(|SC(x^n)|_{n\in\mathbb N})$ is periodic of period $m\cdot (m-1)$.
\end{example}

\begin{example} The aim of the next example is to destroy one possible idea for proving Conjecture~\ref{C:MainConjecture2}.
In the 5-strand braid group with its classical Garside structure, we consider the element $y=\Delta^{-2} 121321432|213214321|121321|232143$.
It is pseudo-Anosov, it satisfies $\inf(y)=-2$, $\sup(y)=2$, it is not rigid, but it is conjugate to the rigid braid $x=12321 . 32143$ with $\inf(x)=0$ and $\sup(x)=2$.
The powers of~$y$ are as follows:
$$\begin{array}{rcl}
y & = & \Delta^{-2}121321432|213214321|121321|232143\\
y^2 & = & \Delta^{0}12321432|2134321|12|213\\
y^3 & = & \Delta^{-2}121321432|213214321|121321|232143|12321432|2134321|12|213\\
y^4 & = & \Delta^{0}12321432|2134321|12|213|12321432|2134321|12|213\\
 & \vdots & 
\end{array}$$
The braid $y^n$ is rigid if and only if $n$ is even. We have $\inf(y^n)=-2$ if $n$~is odd, $\inf(y^n)=0$ if $n$~is even, and $\sup(y^k)=2n$.

We observe that the last factors of $y$ and $y^3$ do not coincide.
This example goes to show that the sequence of pairs $(\iota(y^n), \phi(y^n))$ need not be periodic of period $r(y)$, even when $r(y)$ is the first rigid power of~$y$, and even when the sequence $|SC(x^n)|$ is periodic of period~$r(y)$. (It is, however, conceivable that this sequence is eventually periodic, starting from its first repetition of a pair.)
\end{example} 

The final three examples will illustrate the proof of our main theorem. They take place in~$\calB_4^*$, the 4-strand braid group equipped with the dual Garside-structure of~\cite{BKL}.
Here is a quick reminder of how this structure works. The four punctures are arranged in a circular fashion around the disk, and the Garside element~$\delta$ is given by a counterclockwise cyclic movement of all four punctures in the disk by an angle of $\frac{\pi}{2}$, giving rise to a cyclic permutation of the punctures. 
The divisors of $\delta$ are in bijection with non-crossing partitions of the four punctures -- indeed, given a non-crossing partition, we get a braid by a movement of the punctures which cyclically exchanges the punctures in the same subset. Thus, the divisors of~$\delta$ in $\calB_4^*$ are the trivial element, the six atoms $\west$, $\north$, $\east$, $\south$, $\adiag$ and $\mdiag$, as well as $\se$, $\ne$, $\nw$, $\sw$, $\ns$, $\ew$, and finally $\square=\delta$. We will call $\adiag$ and $\mdiag$ the ``diagonal elements''. Here are some examples of relations between these generators: $\west\cdot\north=\nw$ (whereas the product $\north\cdot\west$ cannot be simplified and is in normal form), $\adiag\cdot\east=\se$, $\adiag\cdot\ew=\square$ $\se\cdot\west=\square$. We also recall the automorphism $\tau(x)=\delta^{-1}x\delta$, which can be interpreted as a $\frac{\pi}{2}$ counterclockwise rotation of the disk: $\tau(\south)=\east$, $\tau(\nw)=\sw$ etc.

\begin{example}\label{E:ConjugationInB4dPer2}
In $\calB_4^*$: we define $x=\mdiag\adiag\north\west\adiag$. Then the sequence $(|SC(x^n)|)_{n\in\mathbb N}$ appears to be periodic of period~2, with $|SC(x)|=7$ and $|SC(x^2)|=7\cdot 20=140$.
\begin{figure}[htb]
\bigfigure{ \begin{center}\begin{tikzpicture}
\node[bigdot,label={above:$x^2=\left(\mdiag\adiag\north\west\adiag\right)^2$}] (x2) at (0,0) {};
\node[bigdot,label={below:$\mdiag\adiag\mdiag\west\south\mdiag\adiag\north\west\adiag$}] (a) at (3.5,0) {};
\node[bigdot,label={above:$\mdiag\adiag\mdiag\adiag\south\east\adiag\north\west\adiag$}] (b) at (7,0) {};
\node[bigdot,label={below:$\mdiag\adiag\mdiag\adiag\mdiag\east\north\north\east\adiag$}] (c) at (10.5,0) {};
\draw[-Stealth,thick,gray] (x2)  to[bend left=10] node[above] {} (a); 
\draw[-Stealth,thick,gray] (a)  to[bend left=10] node[below] {} (x2); 
\draw[-Stealth,thick,gray] (a)  to[bend left=10] node[above] {} (b); 
\draw[-Stealth,thick,gray] (b)  to[bend left=10] node[below] {} (a); 
\draw[-Stealth,thick,gray] (b)  to[bend left=10] node[above] {} (c); 
\draw[-Stealth,thick,gray] (c)  to[bend left=10] node[below] {} (b); 
\end{tikzpicture}\end{center} }
\caption{The conjugacy graph of $x^2$ for $x=\mdiag\adiag\north\west\adiag$}
\label{F:ConjGraphB4dPer2}\end{figure}

The conjugacy diagram of $SC(x)$ consists of only a single vertex (i.e. the only rigid conjugates of $x$ are those obtained by cyclic permutation of its factors and the action of~$\tau$), whereas the diagram of $x^2$ has four vertices. Each of the rightward-pointing arrows in the diagram represents conjugation by~$\west$, each of the leftward-pointing arrows represents conjugation by $\east$. 

\begin{figure}[htb]
\bigfigure{ \begin{center}\begin{tikzpicture}
\node[dot] (u1) at (0,1.5) {};
\node[dot] (u2) at (1,1.5) {};
\node[dot] (d2) at (1,0) {};
\node[dot] (u3) at (2,1.5) {};
\node[dot] (d3) at (2,0) {};
\node[dot] (u4) at (3,1.5) {};
\node[dot] (d4) at (3,0) {};
\node[dot] (u5) at (4,1.5) {};
\node[dot] (d5) at (4,0) {};
\node[dot] (u6) at (5,1.5) {};
\node[dot] (d6) at (5,0) {};
\node[dot] (u7) at (6,1.5) {};
\node[dot] (d7) at (6,0) {};
\node[dot] (u8) at (7,1.5) {};
\node[dot] (d8) at (7,0) {};
\node[dot] (u9) at (8,1.5) {};
\node[dot] (d9) at (8,0) {};
\node[dot] (u10) at (9,1.5) {};
\node[dot] (d10) at (9,0) {};
\node[dot] (u11) at (10,1.5) {};
\node[dot] (d11) at (10,0) {};
\node[dot] (u12) at (11,1.5) {};
\node[dot] (d12) at (11,0) {};
\draw[-Stealth] (u1) to node[above] {$\adiag$} (u2);
\draw[-Stealth] (u2) to node[above] {$\mdiag$} (u3);
\draw[-Stealth] (u3) to node[above] {$\adiag$} (u4);
\draw[-Stealth] (u4) to node[above] {$\north$} (u5);
\draw[-Stealth] (u5) to node[above] {$\west$} (u6);
\draw[-Stealth] (u6) to node[above] {$\adiag$} (u7);
\draw[-Stealth] (u7) to node[above] {$\mdiag$} (u8);
\draw[-Stealth] (u8) to node[above] {$\adiag$} (u9);
\draw[-Stealth] (u9) to node[above] {$\north$} (u10);
\draw[-Stealth] (u10) to node[above] {$\west$} (u11);
\draw[-Stealth] (u11) to node[above] {$\adiag$} (u12);
\draw[-Stealth,red] (u2) to node[near end] {\black{$\west$}} (d2);
\draw[-Stealth] (u3) to node[near end] {$\south$} (d3);
\draw[-Stealth] (u4) to node[near end] {$\east$} (d4);
\draw[-Stealth] (u5) to node[near end] {$\east$} (d5);
\draw[-Stealth] (u6) to node[near end] {$\east$} (d6);
\draw[-Stealth,red] (u7) to node[near end] {\black{$\east$}} (d7);
\draw[-Stealth] (u8) to node[near end] {$\north$} (d8);
\draw[-Stealth] (u9) to node[near end] {$\west$} (d9);
\draw[-Stealth] (u10) to node[near end] {$\adiag$} (d10);
\draw[-Stealth] (u11) to node[near end] {$\north$} (d11);
\draw[-Stealth,red] (u12) to node[near end] {\black{$\west$}} (d12);
\draw[-Stealth] (u2) to node[near start] {\phantom{ww}$\sw$} (d3);
\draw[-Stealth] (u3) to node[near start] {\phantom{ww}$\se$} (d4);
\draw[-Stealth] (u4) to node[near start] {\phantom{ww}$\ne$} (d5);
\draw[-Stealth] (u5) to node[near start] {\phantom{ww}$\ew$} (d6);
\draw[-Stealth] (u6) to node[near start] {\phantom{ww}$\se$} (d7);
\draw[-Stealth] (u7) to node[near start] {\phantom{ww}$\ne$} (d8);
\draw[-Stealth] (u8) to node[near start] {\phantom{ww}$\nw$} (d9);
\draw[-Stealth] (u9) to node[near start] {\phantom{ww}$\nw$} (d10);
\draw[-Stealth] (u10) to node[near start] {\phantom{ww}$\nw$} (d11);
\draw[-Stealth] (u11) to node[near start] {\phantom{ww}$\nw$} (d12);
\draw[-Stealth] (d2) to node[below] {$\mdiag$} (d3);
\draw[-Stealth] (d3) to node[below] {$\adiag$} (d4);
\draw[-Stealth] (d4) to node[below] {$\mdiag$} (d5);
\draw[-Stealth] (d5) to node[below] {$\west$} (d6);
\draw[-Stealth] (d6) to node[below] {$\south$} (d7);
\draw[-Stealth] (d7) to node[below] {$\mdiag$} (d8);
\draw[-Stealth] (d8) to node[below] {$\adiag$} (d9);
\draw[-Stealth] (d9) to node[below] {$\north$} (d10);
\draw[-Stealth] (d10) to node[below] {$\west$} (d11);
\draw[-Stealth] (d11) to node[below] {$\adiag$} (d12);
\draw[-Stealth,dashed] (1.05,2.1) to node[above] {$x$} (5.95,2.1);
\draw[-Stealth,dashed] (6.05,2.1) to node[above] {$x$} (10.95,2.1);
\draw[-] (1,2.05) to (1,2.15);
\draw[-] (6,2.05) to (6,2.15);
\draw[-] (11,2.05) to (11,2.15);
\draw[color=blue,thick] (1.25,1.5) arc [start angle=0, end angle=180, radius=0.25];
\draw[color=blue,thick] (2.25,1.5) arc [start angle=0, end angle=180, radius=0.25];
\draw[color=blue,thick] (3.25,1.5) arc [start angle=0, end angle=180, radius=0.25];
\draw[color=blue,thick] (4.25,1.5) arc [start angle=0, end angle=180, radius=0.25];
\draw[color=blue,thick] (5.25,1.5) arc [start angle=0, end angle=180, radius=0.25];
\draw[color=blue,thick] (6.25,1.5) arc [start angle=0, end angle=180, radius=0.25];
\draw[color=blue,thick] (7.25,1.5) arc [start angle=0, end angle=180, radius=0.25];
\draw[color=blue,thick] (8.25,1.5) arc [start angle=0, end angle=180, radius=0.25];
\draw[color=blue,thick] (9.25,1.5) arc [start angle=0, end angle=180, radius=0.25];
\draw[color=blue,thick] (10.25,1.5) arc [start angle=0, end angle=180, radius=0.25];
\draw[color=blue,thick] (11.25,1.5) arc [start angle=0, end angle=180, radius=0.25];
\draw[color=blue,thick] (2.25,0) arc [start angle=0, end angle=121, radius=0.25];
\draw[color=blue,thick] (3.25,0) arc [start angle=0, end angle=121, radius=0.25];
\draw[color=blue,thick] (4.25,0) arc [start angle=0, end angle=121, radius=0.25];
\draw[color=blue,thick] (5.25,0) arc [start angle=0, end angle=121, radius=0.25];
\draw[color=blue,thick] (6.25,0) arc [start angle=0, end angle=121, radius=0.25];
\draw[color=blue,thick] (7.25,0) arc [start angle=0, end angle=121, radius=0.25];
\draw[color=blue,thick] (8.25,0) arc [start angle=0, end angle=121, radius=0.25];
\draw[color=blue,thick] (9.25,0) arc [start angle=0, end angle=121, radius=0.25];
\draw[color=blue,thick] (10.25,0) arc [start angle=0, end angle=121, radius=0.25];
\draw[color=blue,thick] (11.25,0) arc [start angle=0, end angle=121, radius=0.25];
\draw[densely dotted,color=blue,thick] (1.75,0) arc [start angle=180, end angle=360, radius=0.25];
\draw[densely dotted,color=blue,thick] (2.75,0) arc [start angle=180, end angle=360, radius=0.25];
\draw[densely dotted,color=blue,thick] (3.75,0) arc [start angle=180, end angle=360, radius=0.25];
\draw[densely dotted,color=blue,thick] (4.75,0) arc [start angle=180, end angle=360, radius=0.25];
\draw[densely dotted,color=blue,thick] (5.75,0) arc [start angle=180, end angle=360, radius=0.25];
\draw[densely dotted,color=blue,thick] (6.75,0) arc [start angle=180, end angle=360, radius=0.25];
\draw[densely dotted,color=blue,thick] (7.75,0) arc [start angle=180, end angle=360, radius=0.25];
\draw[densely dotted,color=blue,thick] (8.75,0) arc [start angle=180, end angle=360, radius=0.25];
\draw[densely dotted,color=blue,thick] (9.75,0) arc [start angle=180, end angle=360, radius=0.25];
\draw[densely dotted,color=blue,thick] (10.75,0) arc [start angle=180, end angle=360, radius=0.25];
\node[circle,label={right:$\ldots$}] (dots1) at (-0.1,0.75) {};
\node[circle,label={right:$\ldots$}] (dots2) at (11.1,0.75) {};
\end{tikzpicture}\end{center} }
\caption{The conjugation of $x^2$ by $\west$ for $x=\mdiag\adiag\north\west\adiag$}
\label{F:ConjugationInB4dPer2} \end{figure}
\end{example}
We remark that the conjugacy graphs of $x^2$ can become arbitrarily large, e.g.if we choose $x$ from the family $x=\mdiag\adiag\north\west\adiag \left( \mdiag \adiag \mdiag \adiag \right)^s$ with $s\in\mathbb{N}$.

\begin{example}\label{E:ConjugationInB4dWithInf}
This example is similar to the previous one, but it illustrates the case $\inf(x)\neq 0$. In $\calB_4^*$, we consider $x=\delta \adiag\adiag$. 
\begin{figure}[htb]
\bigfigure{ \begin{center}\begin{tikzpicture}
\node[dot] (u1) at (0,1.5) {};
\node[dot] (u2) at (1.5,1.5) {};
\node[dot] (u3) at (3,1.5) {};
\node[dot] (u4) at (4.5,1.5) {};
\node[dot] (u5) at (6,1.5) {};
\node[dot] (u6) at (7.5,1.5) {};
\node[dot] (u7) at (9,1.5) {};
\node[dot] (u8) at (10.5,1.5) {};
\node[dot] (d2) at (1.5,0) {};
\node[dot] (d3) at (3,0) {};
\node[dot] (d4) at (4.5,0) {};
\node[dot] (d5) at (6,0) {};
\node[dot] (d6) at (7.5,0) {};
\node[dot] (d7) at (9,0) {};
\node[dot] (d8) at (10.5,0) {};
\draw[-Stealth] (u1) to node[above] {$\adiag$} (u2); 
\draw[-Stealth] (u2) to node[above] {$\square$} (u3);
\draw[-Stealth] (u3) to node[above] {$\adiag$} (u4);
\draw[-Stealth] (u4) to node[above] {$\adiag$} (u5);
\draw[-Stealth] (u5) to node[above] {$\square$} (u6);
\draw[-Stealth] (u6) to node[above] {$\adiag$} (u7);
\draw[-Stealth] (u7) to node[above] {$\adiag$} (u8);
\draw[-Stealth,red] (u2) to node {$\textcolor{black}{\west}$\phantom{ww}} (d2); 
\draw[-Stealth] (u4) to node {$\east$\phantom{ww}} (d4);
\draw[-Stealth,red] (u5) to node {$\textcolor{black}{\east}$\phantom{ww}} (d5);
\draw[-Stealth] (u7) to node {$\west$\phantom{ww}} (d7);
\draw[-Stealth,red] (u8) to node {$\textcolor{black}{\west}$\phantom{ww}} (d8);
\draw[-Stealth] (u1) to node[near start] {\phantom{$ww.$}$\nw$} (d2);
\draw[-Stealth] (u2) to node[near start] {\phantom{$ww.$}$\sw$} (d3);
\draw[-Stealth] (u3) to node[near start] {\phantom{$ww.$}$\se$} (d4);
\draw[-Stealth] (u4) to node[near start] {\phantom{$ww.$}$\se$} (d5);
\draw[-Stealth] (u5) to node[near start] {\phantom{$ww.$}$\ne$} (d6);
\draw[-Stealth] (u6) to node[near start] {\phantom{$ww.$}$\nw$} (d7);
\draw[-Stealth] (u7) to node[near start] {\phantom{$ww.$}$\nw$} (d8);
\draw[-Stealth] (d2) to node[below] {$\mdiag$} (d3);
\draw[-Stealth] (d3) to node[below] {$\square$} (d4);
\draw[-Stealth] (d4) to node[below] {$\south$} (d5);
\draw[-Stealth] (d5) to node[below] {$\mdiag$} (d6);
\draw[-Stealth] (d6) to node[below] {$\square$} (d7);
\draw[-Stealth] (d7) to node[below] {$\north$} (d8);
\node[circle,label={right:$\ldots$}] (dots1) at (-0.8,0.75) {};
\node[circle,label={right:$\ldots$}] (dots2) at (10.6,0.75) {};
\draw[color=blue,thick] (1.2,1.5) arc [start angle=180, end angle=90, radius=0.3];
\draw[color=blue,thick] (3.3,1.5) arc [start angle=0, end angle=90, radius=0.3];
\draw[color=blue,thick] (4.8,1.5) arc [start angle=0, end angle=180, radius=0.3];
\draw[color=blue,thick] (5.7,1.5) arc [start angle=180, end angle=90, radius=0.3]; 
\draw[color=blue,thick] (7.8,1.5) arc [start angle=0, end angle=90, radius=0.3]; 
\draw[color=blue,thick] (9.3,1.5) arc [start angle=0, end angle=180, radius=0.3];
\draw[color=blue,thick] (4.8,0) arc [start angle=0, end angle=135, radius=0.3];
\draw[color=blue,thick] (6.3,0) arc [start angle=0, end angle=135, radius=0.3];
\draw[color=blue,thick] (9.3,0) arc [start angle=0, end angle=135, radius=0.3];
\draw[densely dotted,color=blue,thick] (2.7,0) arc [start angle=180, end angle=270, radius=0.3];
\draw[densely dotted,color=blue,thick] (4.8,0) arc [start angle=360, end angle=270, radius=0.3];
\draw[densely dotted,color=blue,thick] (5.7,0) arc [start angle=180, end angle=360, radius=0.3];
\draw[densely dotted,color=blue,thick] (7.2,0) arc [start angle=180, end angle=270, radius=0.3];
\draw[densely dotted,color=blue,thick] (9.3,0) arc [start angle=360, end angle=270, radius=0.3];
\draw[-Stealth,dashed] (1.55,2.1) to node[above] {$x$} (5.95,2.1);
\draw[-Stealth,dashed] (6.05,2.1) to node[above] {$x$} (10.45,2.1);
\draw[-] (1.5,2.05) to (1.5,2.15);
\draw[-] (6,2.05) to (6,2.15);
\draw[-] (10.5,2.05) to (10.5,2.15);
\end{tikzpicture}\end{center} }
\caption{For $x=\delta\cdot\adiag\adiag$, the conjugation of $x^2$ by $\west$ yields $\mdiag \,\delta\, \south \mdiag \,\delta\, \north = \delta^2\,\mdiag\east\adiag\north$.}
\label{F:ConjugationInB4dWithInf} \end{figure}
Then the sequence $(|SC(x^n)|)_{n\in\mathbb N}$ is periodic of period~2, with $|SC(x)|=4$ and $|SC(x^2)|=12$. Indeed, the conjugacy graph of~$x$ has only one vertex, representing the four obvious elements $\delta\,\adiag\adiag$, $\delta\,\adiag\mdiag$, $\delta\,\mdiag\mdiag$, and $\delta\,\mdiag\adiag$. The conjugacy graph of $x^2$ has two vertices: the vertex represented by $x^2=(\delta\,\adiag\adiag)^2=\delta^2\,\mdiag\mdiag\adiag\adiag$, and the one represented by $\west^{-1}x^2\west=\delta^2\,\mdiag\east\mdiag\north$. 
Figure~\ref{F:ConjugationInB4dWithInf} shows the calculation that conjugating $x^2$ by~$\west$ yields $\mdiag \,\delta\, \south \mdiag \,\delta\, \north$. The latter word is almost in normal form: in order to obtain the normal form, it suffices to slide the letters $\delta$ to the start of the word, using the rule $x_i\cdot\delta=\delta\cdot \tau(x_i)$. We find: $\west^{-1}\cdot\delta\,\adiag\adiag\cdot\west = \delta^2\,\mdiag\east\adiag\north$. 
\end{example}

\begin{example}\label{E:ConjugationInB4dPer3}
Still in $\calB_4^*$, we consider $x=\south\south\east\east\north\north\west\west$. 
Then the sequence $(|SC(x^n)|)_{n\in\mathbb N}$ appears to be periodic of period~3, with $|SC(x)|=|SC(x^2)|=3$ and $|SC(x^3)|=32$.
\begin{figure}[htb]
\bigfigure{ \begin{center}\begin{tikzpicture}
\node[dot] (u1) at (0,1) {};
\node[dot] (u2) at (0.5,1) {};
\node[dot] (d2) at (0.5,0) {};
\node[dot] (u3) at (1,1) {};
\node[dot] (d3) at (1,0) {};
\node[dot] (u4) at (1.5,1) {};
\node[dot] (d4) at (1.5,0) {};
\node[dot] (u5) at (2,1) {};
\node[dot] (d5) at (2,0) {};
\node[dot] (u6) at (2.5,1) {};
\node[dot] (d6) at (2.5,0) {};
\node[dot] (u7) at (3,1) {};
\node[dot] (d7) at (3,0) {};
\node[dot] (u8) at (3.5,1) {};
\node[dot] (d8) at (3.5,0) {};
\node[dot] (u9) at (4,1) {};
\node[dot] (d9) at (4,0) {};
\node[dot] (u10) at (4.5,1) {};
\node[dot] (d10) at (4.5,0) {};
\node[dot] (u11) at (5,1) {};
\node[dot] (d11) at (5,0) {};
\node[dot] (u12) at (5.5,1) {};
\node[dot] (d12) at (5.5,0) {};
\node[dot] (u13) at (6,1) {};
\node[dot] (d13) at (6,0) {};
\node[dot] (u14) at (6.5,1) {};
\node[dot] (d14) at (6.5,0) {};
\node[dot] (u15) at (7,1) {};
\node[dot] (d15) at (7,0) {};
\node[dot] (u16) at (7.5,1) {};
\node[dot] (d16) at (7.5,0) {};
\node[dot] (u17) at (8,1) {};
\node[dot] (d17) at (8,0) {};
\node[dot] (u18) at (8.5,1) {};
\node[dot] (d18) at (8.5,0) {};
\node[dot] (u19) at (9,1) {};
\node[dot] (d19) at (9,0) {};
\node[dot] (u20) at (9.5,1) {};
\node[dot] (d20) at (9.5,0) {};
\node[dot] (u21) at (10,1) {};
\node[dot] (d21) at (10,0) {};
\node[dot] (u22) at (10.5,1) {};
\node[dot] (d22) at (10.5,0) {};
\node[dot] (u23) at (11,1) {};
\node[dot] (d23) at (11,0) {};
\node[dot] (u24) at (11.5,1) {};
\node[dot] (d24) at (11.5,0) {};
\node[dot] (u25) at (12,1) {};
\node[dot] (d25) at (12,0) {};
\node[dot] (u26) at (12.5,1) {};
\node[dot] (d26) at (12.5,0) {};
\draw[-Stealth] (u1) to node[above] {$\west$} (u2);
\draw[-Stealth] (u2) to node[above] {$\south$} (u3);
\draw[-Stealth] (u3) to node[above] {$\south$} (u4);
\draw[-Stealth] (u4) to node[above] {$\east$} (u5);
\draw[-Stealth] (u5) to node[above] {$\east$} (u6);
\draw[-Stealth] (u6) to node[above] {$\north$} (u7);
\draw[-Stealth] (u7) to node[above] {$\north$} (u8);
\draw[-Stealth] (u8) to node[above] {$\west$} (u9);
\draw[-Stealth] (u9) to node[above] {$\west$} (u10);
\draw[-Stealth] (u10) to node[above] {$\south$} (u11);
\draw[-Stealth] (u11) to node[above] {$\south$} (u12);
\draw[-Stealth] (u12) to node[above] {$\east$} (u13);
\draw[-Stealth] (u13) to node[above] {$\east$} (u14);
\draw[-Stealth] (u14) to node[above] {$\north$} (u15);
\draw[-Stealth] (u15) to node[above] {$\north$} (u16);
\draw[-Stealth] (u16) to node[above] {$\west$} (u17);
\draw[-Stealth] (u17) to node[above] {$\west$} (u18);
\draw[-Stealth] (u18) to node[above] {$\south$} (u19);
\draw[-Stealth] (u19) to node[above] {$\south$} (u20);
\draw[-Stealth] (u20) to node[above] {$\east$} (u21);
\draw[-Stealth] (u21) to node[above] {$\east$} (u22);
\draw[-Stealth] (u22) to node[above] {$\north$} (u23);
\draw[-Stealth] (u23) to node[above] {$\north$} (u24);
\draw[-Stealth] (u24) to node[above] {$\west$} (u25);
\draw[-Stealth] (u25) to node[above] {$\west$} (u26);
\draw[-Stealth] (d2) to node[below] {$\south$} (d3);
\draw[-Stealth] (d3) to node[below] {$\south$} (d4);
\draw[-Stealth] (d4) to node[below] {$\east$} (d5);
\draw[-Stealth] (d5) to node[below] {$\north$} (d6);
\draw[-Stealth] (d6) to node[below] {$\north$} (d7);
\draw[-Stealth] (d7) to node[below] {$\mdiag$} (d8);
\draw[-Stealth] (d8) to node[below] {$\west$} (d9);
\draw[-Stealth] (d9) to node[below] {$\west$} (d10);
\draw[-Stealth] (d10) to node[below] {$\south$} (d11);
\draw[-Stealth] (d11) to node[below] {$\east$} (d12);
\draw[-Stealth] (d12) to node[below] {$\east$} (d13);
\draw[-Stealth] (d13) to node[below] {$\adiag$} (d14);
\draw[-Stealth] (d14) to node[below] {$\north$} (d15);
\draw[-Stealth] (d15) to node[below] {$\north$} (d16);
\draw[-Stealth] (d16) to node[below] {$\west$} (d17);
\draw[-Stealth] (d17) to node[below] {$\south$} (d18);
\draw[-Stealth] (d18) to node[below] {$\south$} (d19);
\draw[-Stealth] (d19) to node[below] {$\mdiag$} (d20);
\draw[-Stealth] (d20) to node[below] {$\east$} (d21);
\draw[-Stealth] (d21) to node[below] {$\east$} (d22);
\draw[-Stealth] (d22) to node[below] {$\north$} (d23);
\draw[-Stealth] (d23) to node[below] {$\west$} (d24);
\draw[-Stealth] (d24) to node[below] {$\west$} (d25);
\draw[-Stealth] (d25) to node[below] {$\adiag$} (d26);
\draw[-Stealth,red] (u2) to node {\black{$\north$}} (d2);
\draw[-Stealth,red] (u10) to node {\black{$\east$}} (d10);
\draw[-Stealth,red] (u18) to node {\black{$\mdiag$}} (d18);
\draw[-Stealth,red] (u26) to node {\black{$\north$}} (d26);
\node[circle,label={right:$\ldots$}] (dots1) at (-1.1,0.5) {};
\node[circle,label={right:$\ldots$}] (dots2) at (13,0.5) {};
\end{tikzpicture}\end{center} }
\caption{The conjugation of $x^3$ by $\north$ for $x=\south\south\east\east\north\north\west\west$}
\label{F:ConjugationInB4dPer3} \end{figure}
\end{example}


\section{Periodicity of $|SC(x^n)|$}\label{S:Periodicity}

We are now going to prove Theorem~\ref{T:LatticeOfExponents}.

\begin{lemma}\label{L:LatticeOfExponents}
Let $G$ be a Garside group.

(a) Suppose $y_1,y_2$ are rigid elements of~$G$ which have a common power, i.e.there are positive integers $k,l$ such that $y_{1}^{\thinspace k}=y_2^{\thinspace l}$. Then they have a common root which is rigid: there is a rigid element $y_*$ with $y_1=y_*^{\frac{l}{gcd(k,l)}}$ and $y_2=y_*^{\frac{k}{gcd(k,l)}}$.

(b) If $y\in G$ is such that $y^{m}$ and $y^{n}$ are rigid, then 
$y^{\operatorname{gcd}(m,n)}$ is rigid.
\end{lemma}

\begin{proof} (a) If either $k=1$ or $l=1$, then the result is trivial, so we assume that $k,l>1$.
In order to fix notations, we say the Garside normal forms of $y_1$ and $y_2$ are $y_1 = \Delta^p u_1\cdots u_r$ and $y_2=\Delta^q v_1\cdots v_s$.

Since $y_1$  is rigid, the normal form of~$y_1^{\thinspace k}$ is obtained from the $k$-fold repetition
$$
(\Delta^p u_1\cdots u_r) (\Delta^p u_1\cdots u_r) \cdots (\Delta^p u_1\cdots u_r) 
$$
by just sliding the factors $\Delta^p$ to the left. Hence, the normal form of $y_1^{\thinspace k}$ is
$$
y_1^{\thinspace k}= \Delta^{pk}\tau^{p(k-1)}(u_1)\cdots\tau^{p(k-1)}(u_r) \tau^{p(k-2)}(u_1)\cdots\tau^{p(k-2)}(u_r) \cdots\cdots u_1\cdots u_r
$$
In particular, $\inf(y_1^{\thinspace k}) = p\cdot k$, and $\ell_{can}(y_1^{\thinspace k})= r\cdot k$.

In the same way, as $y_2$ is rigid, the normal form of $y_2^{\thinspace l}$ is
$$
y_2^{\thinspace l}= \Delta^{ql}\tau^{q(l-1)}(u_1)\cdots\tau^{q(l-1)}(u_s) \tau^{q(l-2)}(u_1)\cdots\tau^{q(l-2)}(u_s) \cdots\cdots u_1\cdots u_s
$$
As the Garside normal form of $y_1^{\thinspace k}=y_2^{\thinspace l}$ is unique, it follows that $pk=ql$ and $rk=sl$.

We now denote $d=gcd(k,l)$, as well as $a=\frac{k}{d}$ and $b=\frac{l}{d}$, so that $k=ad$ and $l=bd$, where $a$ and~$b$ are coprime. With this notation, $p=\frac{ql}{k}=\frac{qb}{a}$ and $r=\frac{sl}{k}=\frac{sb}{a}$ are integers, which implies that $p$ and~$r$ are divisible by $b=\frac{l}{gcd(k,l)}$.
Similarly, $q$ and~$s$ are divisible by $a=\frac{k}{gcd(k,l)}$.

Now, as the $\Delta$-factors can be moved along the Garside normal form to any desired position (conjugating the simple factors), we can decompose $y_1^{\thinspace k}$ into $k\cdot b$ factors:
$$
y_1^{\thinspace k} = \Delta^\frac{p}{b} Y_1 \cdot \Delta^\frac{p}{b} Y_2 \cdot\ldots\cdot \Delta^\frac{p}{b} Y_{kb}
$$
where each $Y_i$ is a normal form word consisting of $\frac{r}{b}=\frac{s}{a}$ consecutive letters of the normal form representative of~$y_1^{\thinspace k}$, conjugated by a suitable power of~$\Delta$. Also, $y_2^{\thinspace l}$ has a similar decomposition 
$$
y_2^{\thinspace l} = \Delta^\frac{q}{a} Y_1 \cdot \Delta^\frac{q}{a} Y_2 \cdot\ldots\cdot \Delta^\frac{q}{a} Y_{la}
$$
where we recall that $\frac{q}{a}=\frac{p}{b}$ and $la=kb$.

Now consider the finite sequence $\mathcal S=(\Delta^{\frac{p}{b}}Y_1, \Delta^{\frac{p}{b}}Y_2, \ldots, \Delta^{\frac{p}{b}}Y_{kb})$. By construction,
$$
y_1 = \Delta^{\frac{p}{b}}Y_1 \cdots \Delta^{\frac{p}{b}}Y_{b} = 
\Delta^{\frac{p}{b}}Y_{b+1} \cdots \Delta^{\frac{p}{b}}Y_{2b} = \ldots =
\Delta^{\frac{p}{b}}Y_{kb-b+1} \cdots \Delta^{\frac{p}{b}}Y_{kb}
$$
Thus the sequence $\mathcal S$ is periodic of period~$b$. By the same argument applied to the element~$y_2$, the sequence is periodic of period~$a$. Since $a$ and~$b$ are coprime, the sequence is constant. 
Thus $Y_j=Y_1$ for $j=1,\ldots,kb$. We can set $y_* = \Delta^{\frac{p}{b}}Y_1 = \Delta^{\frac{q}{a}}Y_1$, and we obtain
$$
y_1 = y_*^{\thinspace b}  \qquad \text{and}\qquad y_2 = y_*^{\thinspace a}
$$
as desired. Finally, we see that $y_*=\Delta^{\frac{p}{b}}Y_1$ is rigid. Recall that the normal form of $y^k$ is obtained by just concatenating $kb$ copies of the normal form of $y^*$ and sliding the $\Delta$ factors to the left. Thus, the same happens with just two copies of $y_*$, so $y_*$ is rigid. This completes the proof of Lemma~\ref{L:LatticeOfExponents}(a).

In order to prove part~(b), we apply part~(a) in the special case $y_1=y^m$, $y_2=y^n$, $k=n$ and $l=m$. If we denote $d=\gcd(m,n)$, we obtain that it exists a rigid element $y_*$ such that $y^m=y_*^{\frac{m}{d}}$  and $y^n=y_*^{\frac{n}{d}}$. We will show that $y^d$ is rigid by proving that $y^d=y_*$. For that purpose, we remark that by B\'ezout's identity, there are integers $\alpha,\beta$ such that $\alpha m + \beta n = d$. Hence:
$$
  y^d = y^{\alpha m + \beta n} = (y^m)^\alpha (y^n)^{\beta} = 
  (y_*^{\frac{m}{d}})^\alpha (y_*^{\frac{n}{d}})^{\beta} = (y_*)^{\frac{\alpha m + \beta n}{d}} = y_* \vspace{-3mm} 
$$
\end{proof}

\begin{proof}[Proof of Theorem~\ref{T:LatticeOfExponents}] 
Part~(a) of the theorem is an immediate consequence of Lemma~\ref{L:LatticeOfExponents}(b).

In order to prove Theorem~\ref{T:LatticeOfExponents}(b), we define a \emph{primitive} element of $\cup_{n\in\mathbb N} SC(x^n)$ to be an element which cannot be written as a power of any other element of~$\cup_{n\in\mathbb N} SC(x^n)$.
(Note that it is conceivable that, for instance, a primitive element lying in~$SC(x^2)$ might have a rigid third root -- that root would necessarily be non-conjugate to any power of~$x$.)

We claim that two distinct primitive elements $y_1, y_2$ of~$\cup_{n\in\mathbb N} SC(x^n)$ cannot have a common power. Indeed, suppose that $y_1\in SC(x^m)$ and $y_2\in SC(x^n)$ for some $m,n>0$ are primitive elements having a common power $z=y_1^k=y_2^l$. Let us denote $d=\gcd(k,l)$ and $r=\operatorname{lcm}(k,l)$.

We can assume that $x$ is non-trivial, otherwise we could not have two distinct primitive elements. Hence, either $\inf(x)\neq 0$ or $\ell_{can}(x)\neq 0$. Now $z$ is rigid and conjugate to $x^{mk}$ and also to $x^{nl}$, so $\inf(z)=mk\inf(x)=nl\inf(x)$ and $\ell_{can}(z)=mk\ell_{can}(x)=nl\ell_{can}(x)$. Hence, $mk=nl$, it is a multiple of $r$, and we can denote the integer $t=\frac{mk}{r}=\frac{nl}{r}$. 

By Lemma~\ref{L:LatticeOfExponents}(a), there is a rigid element $y_*$ such that $y_1=y_*^{\frac{l}{d}}$ and $y_2=y_*^{\frac{k}{d}}$. This implies that $y_*^{r}=y_*^{\frac{kl}{d}}=y_1^k = y_2^l = z$. But $(x^{t})^r=\left(x^{\frac{mk}{r}}\right)^{r} = x^{mk}$, which is conjugate to $z$. That is, $y_*$ and a conjugate of $x^t$ have a common $r$-th power, $z$. Since, by assumption, roots in $G$ are unique up to conjugacy, it follows that $y_*$ is a rigid conjugate of $x^t$. Hence $y_*\in \cup_{n\in\mathbb N} SC(x^n)$. As $y_1$ and $y_2$ are powers of $y_*$, they cannot be both primitive, and we get a contradiction. This proves the claim that distinct primitive elements cannot have a common power.

Now we observe that the set $\cup_{n\in\mathbb N} SC(x^n)$ consists of the set of all powers of all primitive elements.
It follows that
$$ |SC(x^n)| = \sum_{k|n} |\{ y\in SC(x^k) | y \text{ primitive} \}| $$
Since the sequence $|SC(x^n)|_{n\in\mathbb N}$ is bounded, we deduce that there are only finitely many primitive elements, say $p$. More explicitly, if we enumerate the primitive elements arbitrarily, and we denote $r_i$ the positive integer such that the $i$th primitive element belongs to $SC(x^{r_i})$, then for $r_*=lcm(r_1,\ldots,r_p)$ we have:
$|SC(x^{r_*})| = max_{n\in\mathbb N} |SC(x^n)|$, and $r_*$ is the smallest number with this property. 
Now our sequence is periodic of period~$r_*$.
Thus we have proven Theorem~\ref{T:LatticeOfExponents} and explicitly identified the number~$r_*$.
\end{proof}

Next, we prove the promised implication between conjectures. We recall

{\bf Conjecture~\ref{C:MainConjecture2} } {\sl For any Garside group~$G$, there exists an integer~$\calR_G$ such that for any element $y$ of~$G$ which has a rigid power and which is conjugate to a rigid element, $r(y)\leqslant \calR_G$.} 
\smallskip

{\bf Proposition~\ref{P:ImplicConjectures}} 
{\sl If Conjecture~\ref{C:MainConjecture2} holds for some Garside group~$G$ where roots are unique up to conjugacy, then our Main Conjecture~\ref{C:MainConjecture} holds for~$G$.}

\begin{proof}
Suppose $G$ is a Garside group with unique roots up to conjugacy and in which Conjecture~\ref{C:MainConjecture2} holds. Let us study the coefficients $r_1,\ldots,r_p$ in the proof of Theorem~\ref{T:LatticeOfExponents}(b) above. Let's say the $i$th primitive element is $t^{-1} x^{r_i} t\in SC(x^{r_i})$. Then $r_i = r(t^{-1} x t)\leqslant \calR_G$. Thus we have a uniform bound on the coefficients $r_1,\ldots,r_p$, i.e.\ on ``the lowest floor where new rigid elements appear''.
So, to summarize, $G$ has the following property which, for later reference, we call Property~$(*)$: 

\begin{notation}
We say a Garside group~$G$ satisfies \emph{Property (\negthinspace$*$\negthinspace)} if in~$G$ roots are unique up to conjugacy and there exists a number~$\calR_G$ such that for any rigid element~$x$ all the primitive elements of $\cup_{n\in\mathbb N} SC(x^n)$ lie in the finite set $\cup_{n\leqslant \calR_G} SC(x^n)$. 
\end{notation}

We have seen in the proof of Theorem~\ref{T:LatticeOfExponents}(b) that for any Garside group with Property~$(*)$, the sequence $(|SC(x^n|)_{n\in\mathbb N}$ is periodic of period $lcm(r_1,\ldots,r_p)$ (where $p$ is the number of primitive elements). Thus for any rigid element~$x$ we obtain the bound of $lcm(1,2,3,\ldots,\calR_G)$ on the period of the sequence $|SC(x^n)|$, and this bound depends only on the underlying group~$G$ and its Garside structure.
\end{proof}


\section{Proof of Theorem~\ref{T:main}}\label{S:Proof}

\subsection{Circular Garside groups}\label{SS:ProofForCircularGroups}
Our aim now is to prove Theorem~\ref{T:main}(a). We recall from \cite{GarnierRoots} the definition of the \emph{circular} Garside groups. We fix two positive integer parameters $m,\ell$, along with an alphabet $\{a_0,\ldots,a_{m-1}\}$, where we see the indices in $\mathbb{Z}/m\mathbb{Z}$. For $i\in \mathbb{Z}$ and $p\in \mathbb{Z}_{\geqslant 1}$, we define $s(i,p)$ as the product 
\[s(i,p) = a_i\cdots a_{i+p-1}\]
We define the \emph{circular monoid} $M(m,\ell)$ by generators and relations:
\[M(m,\ell)=\langle a_0,\ldots,a_{m-1}~|~\forall i\in \intv{0,m-1},~ s(i,\ell)=s(i+1,\ell)\rangle^+.\]
This monoid is a Garside monoid with Garside element $\Delta=s(0,\ell)$. Its attached Garside group is denoted by $G(m,\ell)$ and called a \emph{circular group}.

By definition, the presentation of $G(2,e)$ is the Artin presentation of the Artin group of type $I_2(e)$, whereas the presentation of $G(e,2)$ is the dual presentation of the same group. In particular, we have $G(2,3)=\mathcal{B}_3$ and $G(3,2)=\mathcal{B}_3^*$.

We list a few remarkable features of circular groups which we will use often:
\begin{itemize}
\item The nontrivial simple elements are exactly the elements $s(i,p)$ with $1\leqslant p\leqslant \ell$.
\item If $s(i,p)$ is a simple element which is nontrivial and different from $\Delta$ (i.e. $1\leqslant p<\ell)$, then there is only one way to write $s(i,p)$ as a product of atoms in $M(m,\ell)$. We can then define the \emph{first letter} of $s(i,p)$ as $I(s(i,p)):=a_i$, and the \emph{last letter} as $F(s(i,p)):=a_{i+p-1}$. Given an element $x\in G(m,\ell)$ which is not a power of $\Delta$, we denote $I(x)=I(\iota(x))$ and $F(x)=F(\varphi(x))$.
\item By \cite[Lemma 2.5]{GarnierRoots}, if $s(i,p),s(i',p')$ are two simple elements distinct from $\Delta$, then the product $s(i,p)s(i',p')$ is in Garside normal form if and only if $i+p\not\equiv i'[m]$. In the case where $i+p\equiv i'[m]$, then the Garside normal form of $s(i,p)s(i',p')$ is as follows:
\[s(i,p)s(i',p')=\begin{cases} s(i,p+p')&\text{if }p+p'<\ell,\\ \Delta&\text{if }p+p'=\ell,\\ \Delta s(i+\ell,p+p'-\ell)&\text{if }p+p'>\ell.\end{cases}\]
Note that if $s(i,p)s(i',p')\neq \Delta$, then $I(s(i,p)s(i',p'))=I(s(i,p))$ and $F(s(i,p)s(i',p'))=F(s(i',p'))$. 
\item If $x,y\in G(m,\ell)$ are not powers of $\Delta$, then the Garside normal form of $xy$ is the concatenation of the Garside normal forms of $x$ and of $y$ if and only if the product $F(x)I(y)$ is left-weighted. 
\item For any integer~$k$, the only rigid conjugate of $\Delta^k$ is $\Delta^k$ itself. In fact, $\Delta^k$ is actually the only element in its super-summit set. To see this, one can consider the length morphism $f\co G(m,\ell)\to \Z$ sending $s(i,p)$ to $p$. Since $f(s(i,p))<f(\Delta)$ if $p<\ell$, we have $f(x)\leqslant \ell\sup(x)$, with equality if and only if $x$ is a power of~$\Delta$. If $x$ lies in the super-summit set of~$\Delta^k$, then $f(x)=f(\Delta^k)=\ell k$ since $x$ is conjugate to~$\Delta^k$, and $\sup(x)=\sup(\Delta^k)=k$. Thus $x$ must be a power of~$\Delta$.
\item As shown in~\cite[Theorem 2.16]{GarnierRoots}, roots in circular groups are unique up to conjugacy.
\end{itemize}

\begin{lemma}\label{L:ProductInCircular}
Let $\alpha,\beta\in G(m,\ell)$ with canonical lengths $r$ and $s$, respectively, and suppose that $\alpha,\beta$ and $\alpha\beta$ are not powers of $\Delta$. If $r\geqslant s$ then $I(\alpha\beta)=I(\alpha)$. If $r\leqslant s$ then $F(\alpha\beta)=F(\beta)$.
\end{lemma}
\begin{proof}
Let $x,y,z,t$ be simple elements such that the products $xy$ and $zt$ are left-weighted and that $yz\neq \Delta$. We compute the Garside normal form of $xyzt$:
\begin{itemize}
\item If $yz$ is left-weighted, then the desired normal form is $x|y|z|t$ by definition.
\item If $yz=u$ is a simple element distinct from $\Delta$, then $F(u)=F(z)$ and $I(u)=I(y)$. This implies that the products $xu$ and $ut$ are both left-weighted. The desired Garside normal is then $x|u|t$.
\item If $yz=\Delta u$ where $u$ is a simple element distinct from $\Delta$, then $F(u)=F(z)$ and the product $ut$ is left-weighted. Moreover, we have $I(u)=a_{i+\ell}$ where $a_i=I(y)$. This implies that the product $\tau(x)u$ is left-weighted since $F(\tau(x))=a_{j+\ell}$, where $a_j=F(x)$. The desired normal form is then $\Delta\tau(x)|u|t$
\end{itemize}

Now, suppose that $r\geqslant s$. Let $\alpha=\Delta^p a_1\cdots a_r$ and $\beta=\Delta^q b_1\cdots b_s$ be written in left normal form. By hypothesis, $r>0$ and $s>0$. We set
\[h:=\max (\{0\}\cup \{i\in\intv{1,s}~|~ \tau^{q+i-1}(a_{r-i+1})b_i=\Delta\}).\]
By definition of $h$, we have $\tau^q(a_{r-h+1})\cdots \tau^q(a_r)b_1\cdots b_h=\Delta^h$. This implies that
\begin{align*}
\alpha\beta&=\Delta^pa_1\cdots a_r\Delta^qb_1\cdots b_s\\
&=\Delta^{p+q}\tau^q(a_1)\cdots \tau^q(a_r)b_1\cdots b_s\\
&=\Delta^{p+q}\tau^q(a_1)|\cdots |\tau^q(a_{r-h})\Delta^hb_{h+1}|\cdots |b_s\\
&=\Delta^{p+q+h}\tau^{q+h}(a_1)|\cdots |\tau^{q+h}(a_{r-h})b_{h+1}|\cdots |b_s.
\end{align*}

We cannot have $h=r=s$ since this would imply that $\alpha\beta=\Delta^{p+q+h}$ is a power of $\Delta$. 
\begin{itemize}
\item If $h=s$, then $s<r$ and the left normal form of $\alpha\beta$ is $\Delta^{p+q+h}\tau^{q+h}(a_1)\cdots \tau^{q+h}(a_{r-h)}$. Since $r-h\geqslant 1$, we have $I(\alpha\beta)=I(\alpha)$.
\item If $h<s\leqslant r$, then we can apply the first part of the proof to the product\\ $\tau^{q+h}(a_{r-h-1})|\tau^{q+h}(a_{r-h})b_{h+1}|b_{h+2}$. In each case we obtain that $I(\alpha\beta)=I(\alpha)$.
\end{itemize}

Now, the last claim follows by symmetry, taking into account that the left and right normal forms of an element in a circular group coincide, up to sliding the $\Delta$ factors to the left or to the right.
\end{proof}

We are now ready to prove:

{\bf Theorem~\ref{T:main}(a)} \; 
{\sl Conjecture~\ref{C:MainConjecture2} and the Main Conjecture~\ref{C:MainConjecture} hold in the circular group $G(m,\ell)$. 
Specifically, $\mathcal P_{G(m,\ell)}=\{1\}$ and $\calR_{G(m,\ell)}=\frac{m}{m\wedge \ell}$.}
\medskip

\begin{proof}
We will first prove Conjecture~\ref{C:MainConjecture2}. 
Let $x$ be a rigid element of~$G(m,\ell)$, and $y$ a conjugate of $x$ with a rigid power. 

First case: we suppose that some power of $y$ belongs to $\langle \Delta\rangle$, say $y^p=\Delta^q$. As we noted in the beginning of this section, we have $SC(\Delta^q)=\{\Delta^q\}$. Since $x$ is a rigid conjugate of $y$, the element $x^p$ is a rigid conjugate of $y^p=\Delta^q$, but there is a unique rigid conjugate of $\Delta^q$, hence $x^p=\Delta^q$. Now, if a rigid element $x$ is not a power of $\Delta$, no power of $x$ can be a power of $\Delta$ (we have $\ell_{can}(x^k)=k\ell_{can}(x)$ so $\ell_{can}(x^k)$ can never be $0$). It follows that $x$ must be a power of $\Delta$. Hence $x=\Delta^t$ (where $tp=q)$.

Now, $y$ is conjugate to $x=\Delta^t$, so $y=\alpha^{-1}\Delta^t\alpha$ for some $\alpha$. Since $z=\Delta^{m/m\wedge \ell}$ is central in $G(m,\ell)$ by \cite[Corollary 2.11]{GarnierRoots}, the element $x^{m/m\wedge \ell}=z^t$ is also central, and thus $y^{m/m\wedge \ell}=x^{m/m\wedge \ell}=\Delta^{tm/m\wedge \ell}$ is rigid.   

Second case: we suppose now that no power of $y$ belongs to $\langle \Delta\rangle$. We claim that $I(y^n)=I(y)$ and $F(y^n)=F(y)$ for every $n>0$. The proof will be by induction; the case $n=1$ is trivial, so let us assume that $n>1$ and that the claim holds for smaller values of $n$.

We have $y^n=y^{n-1}y=yy^{n-1}$. By induction hypothesis, $I(y^{n-1})=I(y)$ and $F(y^{n-1})=F(y)$. Since $y,y^{n-1}$ and $y^n$ are not powers of $\Delta$, we can apply Lemma \ref{L:ProductInCircular} to obtain that, if $\ell_{can}(y^{n-1})\geqslant \ell_{can}(y)$, then $I(y^n)=I(y^{n-1}y)=I(y^{n-1})=I(y)$ and $F(y^n)=F(yy^{n-1})=F(y^{n-1})=F(y)$. Similarly, if $\ell_{can}(y^{n-1})\leqslant \ell_{can}(y)$, then $I(y^n)=I(yy^{n-1})=I(y)$ and $F(y^n)=F(y^{n-1}y)=F(y)$. This covers all cases, so the claim is shown.

Recall that an element $\alpha\in G(m,\ell)$ is rigid if and only if $F(\alpha)$ and $I(\alpha)$ are not consecutive atoms. In our case, we know that for some $n>0$ such that $y^n$ is rigid, the atoms $F(y^n)$ and $I(y^n)$ are not consecutive. But since $I(y^n)=I(y)$ and $F(y^n)=F(y)$, this implies that $y$ is also rigid.

Thus the proof of Conjecture~\ref{C:MainConjecture2} is complete, with $\calR_{G(m,\ell)}=\frac{m}{m\wedge \ell}$. Since roots in circular groups are unique up to conjugacy, this implies, in principle, the truth of Main Conjecture~\ref{C:MainConjecture} in this framework (by Proposition~\ref{P:ImplicConjectures}). However, it does not yield the desired bound of $\mathcal P_{G(m,\ell)}=\{1\}$. 


In order to obtain this bound, we take $x$ to be a rigid element of~$G(m,\ell)$, and we have to show that $|SC(x^n)|$ does not depend on~$n$. If $x$ is a power of~$\Delta$, then this has already been discussed in the beginning of this section. Assume now that $x$ is not a power of $\Delta$. By Lemma \ref{L:PowersAndConjGraphs}, we know that the map $\pi^n:y\mapsto y^n$ induces an injection from $SC(x)$ to $SC(x^n)$. It is then sufficient to show that $\pi^n$ is  also surjective. Let $z\in SC(x^n)$, and let $\alpha$ be such that $\alpha x^n\alpha^{-1}=z$. The element $z'=\alpha x\alpha^{-1}$ is a conjugate of the rigid element $x$ which admits a rigid power $z$. Since $x$ is not a power of $\Delta$, we are in the second case of the first part of the proof, and thus $z'$ is rigid. We then have $z'\in SC(x)$ and $\pi^n(z')=z$ as we wanted to show.
\end{proof}

An interesting byproduct of the above proof is that the only case in which $y$ (conjugate to a rigid element) is not rigid but has a rigid power happens when $y$ is a non-rigid conjugate to a power of $\Delta$. Moreover, we showed that $r(y)\leqslant \frac{m}{m\wedge \ell}$ in this case. Here is one more observation:

\begin{lemma}
The bound $\calR_{G(m,\ell)}=\frac{m}{m\wedge \ell}$ is optimal.
\end{lemma}
\begin{proof}
Consider the element $y=s(1,\ell-1)a_0=s(1,\ell-1)s(0,1)$. It is equal to $a_0^{-1}\Delta a_0$, and we claim that $r(y)=\frac{m}{m\wedge \ell}$. 

If $m$ divides $\ell$, then $\frac{m}{m\wedge \ell}=1$. We have $1+\ell-1=\ell\equiv 0[m]$, and thus $y=s(1,\ell)=\Delta$. If $m$ does not divide $\ell$, then $\ell\not\equiv 0[m]$ and the Garside normal form of $y$ is $s(1,\ell-1)|s(0,1)$. Since $s(0,1)s(1,\ell-1)=\Delta$, we have $y^n=s(1,\ell-1)\Delta^{n-1}s(1,0)=\Delta^{n-1}s(1+\ell(n-1),\ell-1)s(1,0)$ for $n>0$. The element $y^n$ is then rigid if and only if it is a power of $\Delta$, which is itself equivalent to the product $s(1+\ell(n-1),\ell-1)s(1,0)$ not being left-weighted. This is again equivalent to
\[1+\ell(n-1)+\ell\equiv 1[m]\; \Leftrightarrow \;  n\ell\equiv0[m].\]
The smallest $n$ for which this holds is $n=\frac{m\vee \ell}{\ell}=\frac{m}{m\wedge \ell}$ as we wanted to show.
\end{proof}


\subsection{The four-strand braid group with the dual Garside structure}\label{SS:ProofForB4*}

We now turn our attention to~$\calB_4^*$, the 4-strand braid group with its dual Garside structure. In this group, the simple elements correspond to non-crossing partitions of $4$ elements, that can be represented as
$$
\{\triv\} \cup \{\west,\north,\east,\south,\mdiag,\adiag\} \cup
\{\sw,\nw,\ne,\se,\ns,\ew\} \cup \{\square\}.
$$
Here, $\triv$ is the trivial element, the second subset is the set of atoms, and the third one is the set of simple elements of weight~$2$, i.e.\ elements that can be written as products of two atoms (in several different ways, e.g.\ $\sw = \west\mdiag = \mdiag\south = \south\west$); 
finally the Garside element $\delta=\square$ has weight~$3$, i.e. it and can be written as a product of three atoms. There are 16 different ways to do so: $\square=\west\east\mdiag=\west\north\east=\north\east\south=...$.

\begin{lemma}\label{L:NoMixing}
Suppose $G$ is a Garside group, equipped with a Garside structure where the Garside element is of weight~3. Suppose the normal form of a rigid element $x$ of~$G$ contains a letter of weight~$1$ \emph{and} a letter of weight~$2$. Then the conjugacy graph of $x$ (and of $x^n$ for any integer~$n$) consists of only one vertex.
\end{lemma}

\begin{proof} Possibly after applying the cycling operation a few times, we can assume that the final factor of~$x$ is has weight~$2$. This implies that $\partial \varphi(x)$ is an atom. Now we recall that any gray arrow in the conjugacy graph leaving the vertex represented by~$x$ is given by conjugation by an element of 
$$
\mathcal C_x=\{c\in G | 1\prec c\prec \partial\phi(x) \}.
$$
Hence, in our situation, this set is empty. We conclude that no gray arrow can leave the vertex represented by~$x$.

Similarly, after some cycling, the initial factor of~$x$ is of weight~$1$, that is, $\iota(x)$ consists of only one atom, and no black arrow can exit the vertex represented by~$x$. 

Since the conjugacy graph of~$x$ is known to be connected \cite{Birman-Gebhardt-GM2}, we conclude that it has only one vertex.
\end{proof}

\begin{remark}\label{R:OnlyWeight1}
Lemma~\ref{L:NoMixing} is very useful for proving Conjecture~\ref{C:MainConjecture} (but not Conjecture~\ref{C:MainConjecture2}) in Garside groups where the Garside element is of weight~$3$. Indeed, when studying the sizes of conjugacy graphs in such groups, we can restrict our attention to rigid elements~$x$ whose normal form contains no letters of weight~2, i.e. only $\Delta$ and letters of weight~1.
(The study of elements whose normal form contains no letters of weight~1 can be reduced to the opposite case -- indeed, if $x$ is rigid and has no letters of weight~1, then $x^{-1}$ is rigid and has no letters of weight~2, and $SC(x^{-1})$ is the set of inverses of the elements in $SC(x)$. Since the operation of taking inverses of rigid elements sends orbits under cycling and $\tau$ to orbits under cycling and $\tau$, it suffices to show the result for $x^{-1}$.)
\end{remark}

We need some more general theory.
The theme of this paper is that, after conjugating a rigid braid by a prefix of the first factor (or by a prefix of the complement of the last factor), we may obtain a braid which is not rigid, but which becomes rigid if we raise it to some power. The following lemma places a bound on the required power. Indeed, a bound is given by the number of strict, non-trivial prefixes of the complement of any of the factors of~$x$:

\begin{lemma}[Limit on rigid power]\label{L:LimitOnPower}
Suppose $x$ is a rigid element of a Garside group~$G$. Consider the set 
$$\mathcal C_x=\{c\in G | 1\prec c\prec \partial\phi(x) \}$$
Suppose that for some element $c_*$ of~$\mathcal C_x$ the conjugate $c_*^{-1}x c_*$ has a rigid power. Let $\rho$ be the smallest such power, i.e. $\rho=r(c_*^{-1}x c_*)$ (using Notation~\ref{N:SmallestRigidPower}).
Then there is a subset $\tilde{\mathcal C}$ of~$\mathcal C_x$ which has precisely $\rho$ elements such that for every $c$ belonging to~$\tilde{\mathcal C}$, the element $c^{-1}x^\rho c$ is rigid.
In particular, $\rho\leqslant |\mathcal C_x|$.
\end{lemma}

There is, of course, an equivalent statement for conjugations of~$x$ by prefixes of~$\iota(x)$. 

\begin{example}\label{E:BoundOnPower}
(a) In Example~\ref{E:SimpleExInB4classical}, there are two different conjugations of $x^2$, visible as red arrows in Figure~\ref{F:ConjugationInB4easy}; indeed, $x^2$ can be conjugated by~$1$ to the rigid braid $21.12.2132.2132.23.32$, and by~$3$ to $2132.23.32.21.12.2132$, which is a cyclic conjugate of the previous word. These two conjugations are indicated by the label ``$\times 2$'' on the arrow originating in the vertex~$x^2$ in Figure~\ref{F:ConjGraphB4easy}.

(b) In Example~\ref{E:ConjugationInB4dPer2}, there are two different conjugations of~$x^2$, one by $\west$ and one by $\east$ -- again, these two are visible as red arrows in Figure~\ref{F:ConjugationInB4dPer2}. 

(c) Exactly the same behaviour as in Example~\ref{E:ConjugationInB4dPer2}  can be observed in Example~\ref{E:ConjugationInB4dWithInf} -- see Figure~\ref{F:ConjugationInB4dWithInf}.

(d) In Example~\ref{E:ConjugationInB4dPer3}, the braid $x^3$ can be conjugated by three different gray arrows (corresponding to conjugations by $\north$, $\east$ and $\mdiag$) to three braids that are cyclic conjugates of each other.

(e) In Example~\ref{E:TheBigOne} (Figure~\ref{F:TheBigOne}), whenever we have an arrow from a vertex that appears in the conjugacy graph of~$x^n$ to a vertex that appears in the conjugacy graph of~$x^N$, with $n|N$, then the corresponding edge is labelled $\times \frac{N}{n}$.
\end{example}

\begin{proof}[Proof of Lemma~\ref{L:LimitOnPower}]
We look at the domino diagram for calculating the normal form of $c_*^{-1}x^\rho c_*$. Let us denote by $\ell=\ell_{can}(x)$, the canonical length of~$x$. 

We start with the $\rho\cdot \ell$ horizontal arrows whose labels spell out the normal form of~$x^\rho$, see Figure~\ref{F:DominoDiagramInf0} and Figure~\ref{F:ConjugationInB4easy} for an example. (We have to take $\rho\cdot (\ell+1)$ arrows if $\inf(x)\neq 0$, see Figure~\ref{F:DominoDiagramInfNeq0} and Figure~\ref{F:ConjugationInB4dWithInf} for an example.) Also, we have a vertical arrow labelled $c_*=c_{\rho\cdot \ell}$ at the right end of the diagram.

We then construct the full domino diagram from right to left, calculating $c_{\rho\cdot \ell-1}, c_{\rho\cdot \ell-2}$ etc. We will be interested in the subsequence $c_{\rho\cdot\ell}, c_{(\rho-1)\cdot\ell},\ldots,c_{2\cdot\ell}, c_{\ell}, c_{0}$, corresponding to the red arrows in our examples. By minimality of~$\rho$, the terms of this subsequence are pairwise distinct, except that $c_{\rho\cdot\ell}=c_0$. These are our $\rho$ distinct conjugators. We remark that the $\rho$ elements $c_0^{-1}x^\rho c_0, \ldots, c_{\rho\cdot\ell}^{-1}x^\rho c_{\rho\cdot\ell}$ are all in the same orbit under the cycling operation. Since the last of them, $c_{\rho\cdot\ell}^{-1}x^\rho c_{\rho\cdot\ell} = c_*^{-1} x^\rho c_*$, is rigid by hypothesis, they are all rigid.
\end{proof}

\begin{remark}
Another way to show Lemma~\ref{L:LimitOnPower} is to notice that the elements $c_{\ell},c_{2\cdot \ell},c_{3\cdot \ell},\ldots$ are the \emph{iterated $\ell$-transports} of $c$ for cycling corresponding to the rigid element $x^\rho$ (see~\cite{Gebhardt}, where they are defined and denoted $c^{(\ell)}, c^{(2\cdot \ell)}, c^{(3\cdot \ell)},\ldots$) Since $c\preccurlyeq \partial\varphi(x^\rho)$, and iterated transport preserves the prefix order and sends the complement of the final factor of an element to the complement of the final factor of its cycling, all the iterated $\ell$-transports of $c$ belong to $\mathcal C_x$. The first iterated $\ell$-transport which equals $c$ corresponds to the smallest power of $y$ which is rigid, so there are exactly $\rho$ distinct  elements, all of them in $\mathcal C_x$.
\end{remark}

\begin{lemma}\label{L:SumOfEdgeLabels}
Let $x$ be a rigid element of~$G$. Look at any vertex represented by some element $y$ in the conjugacy graph of $x^n$, for any~$n$. Look at the labels of the outgoing gray edges from that vertex, including those labelled ''$\times 1$'' (for which by convention we didn't write down the label). Then the sum of those labels is at most $|\mathcal{C}_y|$, where $\mathcal{C}_y = \{c\in G | 1\prec c\prec \partial\phi(y) \}$.
\end{lemma} 

\begin{proof}
This is simply because of the pigeonhole principle: the outgoing edges represent disjoint sets of distinct conjugating elements, which all lie in~$\mathcal{C}_y$.
\end{proof}

Let us now concentrate on the braid group $\calB_4^*$. 
The sliding circuit sets and conjugacy diagrams in the case $\calB_4^*$ were already studied in \cite{CalvezWiestConjugacyB4}.

Because of 
Remark~\ref{R:OnlyWeight1}, when trying to prove our Main Conjecture~\ref{C:MainConjecture} for $\calB_4^*$, we can restrict our attention to rigid elements~$x$ whose normal form contains no letters of weight~2, but only letters of weight~1, and $\delta$.
Any rigid conjugate~$y$ of any power of such an element~$x$ has the same property: it has no letters of weight~$2$. Indeed, any power of $x$ has a normal form $\Delta^p z_1\cdots z_r$ with all $z_i$ having weight 1. Hence, if $y$ is a rigid conjugate of that element, its normal form must be $\Delta^p y_1\cdots y_r$ (same infimum and same canonical length). Since relations are homogeneous in $\calB_4^*$, the word length of $z_1\cdots z_r$ and $y_1\cdots y_r$ must coincide, therefore all letters $y_j$ must have weight 1. In particular, $\iota(y)$ has no strict nontrivial prefixes. Thus 

\begin{observation}\label{O:NoWeight2}
Suppose $x\in \calB_4^*$ is rigid, and contains no letter of weight~2. Then in the conjugacy graph of~$x^n$, for any $n$, all vertices correspond to braids which are also rigid and without letters of weight~$2$. Moreover, there are only gray arrows, no black ones.
\end{observation}

Another key observation is that the presence of ``diagonal'' letters $\mdiag$ ou $\adiag$ imposes strong restrictions on the possible gray arrows. Indeed, diagonal letters have only two strict nontrivial prefixes to their complement, whereas non-diagonal letters have three. For instance:
$$
\{c\in G |  1\prec c\prec \partial\mdiag\} = \{\north, \south\} \text{ whereas } 
\{c\in G | 1\prec c\prec \partial\west\} = \{\north, \mdiag, \east\}
$$

Let us first prove our Main Conjecture~\ref{C:MainConjecture} in the special case where some rigid conjugate of some power of~$x$ contains no diagonal letter:

\begin{lemma}\label{L:NoDiagonals}
Let $x$ be a rigid element of~$\calB_4^*$ containing no letters of weight~2. Suppose that the normal form of some rigid conjugate of some power of~$x$ contains no diagonal letters ($\mdiag$ or $\adiag$). Then for every nonzero integer~$n$, the conjugacy graph of~$x^n$ has diameter at most~3.
\end{lemma}

\begin{proof}
Suppose that a rigid conjugate $y$ of $x^m$ contains no diagonal letter in its normal form, and consider some $n>0$.
By Lemma~\ref{L:PowersAndConjGraphs}, the conjugacy graph of~$x^n$ is contained in the conjugacy graph of $x^{4mn}$, which is the conjugacy graph of $y^{4n}$, whose infimum is congruent to 0 modulo 4. 
The conjugacy graphs of dual 4-strand braids with infimum 0, all of whose letters are of weight 1 and non-diagonal, were studied in detail in~\cite{CalvezWiestConjugacyB4}. It turns out that the conjugacy graphs have at most 6 vertices. 
It is also a fact, not explicitly stated in~\cite{CalvezWiestConjugacyB4} but an immediate consequence of the last paragraph of the proof Proposition~5.3 of that paper, that the diameter of the conjugacy graph is at most~3.
\end{proof}

To summarize, if $x$ is a rigid braid with no letters of weight 2, then the conjugacy graph has only gray arrows, which can only be labelled ``$\times 1$'', ``$\times 2$'', or ``$\times 3$'', by Lemma~\ref{L:SumOfEdgeLabels}. Now suppose that $y$ and $z$ are conjugates of $x$ having rigid powers that represent adjacent vertices in the conjugacy graph of some power of $x$, and suppose that $r(y)<r(z)$. Applying Lemma~\ref{L:LimitOnPower} to the conjugacy graph of $y^{r(y)}$, it follows that $r(z)\leqslant 3\cdot r(y)$. Moreover, if some rigid conjugate of some power of $x$ has no diagonal letters, the conjugacy graph of every power of $x$ has diameter~$d$ at most~3. 

In particular, for any conjugate of $x$ with a rigid power, we have $r(y)\leqslant 3^d\leqslant 3^3=27$.
Using the language of Section~\ref{S:Periodicity}, all the primitive elements of $\cup_{n\in\mathbb N} SC(x^n)$ lie in $\cup_{n\leqslant 27} SC(x^n)$.

We now turn to the second case, where every rigid conjugate of every power of~$x$ contains a diagonal letter.

\begin{lemma}\label{L:NoGrayArrows}
Let $x$ be a rigid element of~$\calB_4^*$ containing no letters of weight~2. Suppose that in the conjugacy graph of~$x^n$ (for some $n\in\mathbb N$) there is a gray arrow emanating from a vertex which represents an element~$y\in SC(x^n)$ whose normal form contains a diagonal letter $\mdiag$ or $\adiag$. Then for any positive integer~$d$, in the conjugacy graph of~$x^{n\cdot d}$, there are no gray arrows emanating from the vertex~$y^d$ other than those inherited via $\pi^d$ from the conjugacy graph of~$x^n$. 
\end{lemma}
\begin{proof}
After cycling, we can assume that the last letter of~$y$ is a diagonal letter. In that case, $\partial \phi(y)$ has only two strict prefixes, so only two gray arrows (counted with multiplicity) can emanate from the vertex~$y$ in the conjugacy graph of~$x^n$, or from the vertex~$y^d$ in the conjugacy graph of $x^{n\cdot d}$. 
If $d$ was the smallest integer so that a new gray arrow appeared, emanating from the vertex $y^d$ in the conjugacy graph of~$x^{n\cdot d}$, then by Lemma~\ref{L:LimitOnPower}, there are $d$ copies of this gray arrow. Also, one further gray arrow is inherited via $\pi^d$ from the conjugacy graph of~$x^n$, by hypothesis. In summary, we'd have $d+1$ gray arrows emanating from the vertex~$y^d$. Thus $d=1$.
\end{proof}

As an immediate consequence of Observation~\ref{O:NoWeight2} and Lemma~\ref{L:NoGrayArrows} we have:

\begin{lemma}\label{L:NoFurtherPowers}
Let $x$ be a rigid element of~$\calB_4^*$ containing no letters of weight~2. Suppose that the normal form of every rigid conjugate of every power of~$x$ contains a diagonal letter ($\mdiag$ or $\adiag$). Suppose that $n$ and $N$ are two positive integers with $n|N$. If the conjugacy graph of~$x^n$ contains more than one vertex, then the conjugacy graph of $x^N$ is isomorphic to the conjugacy graph of~$x^n$.
\end{lemma}

We are now ready to prove our Main Theorem~\ref{T:main}(b) for elements $x$ which are rigid, have no letters of weight~2, and for which every rigid conjugate of any power of~$x$ contains a diagonal letter. 

Let us fix some power~$n$, and look at the conjugacy graph of~$x^n$. There are two possibilities.

Either the conjugacy graph of~$x^n$ has only one vertex. Let $d$ be the smallest integer such that the conjugacy graph of~$x^{d\cdot n}$ has more than one vertex. (If there exists none, then the proof is complete.) Let $c_*$ be a conjugating element, representing a gray arrow in that conjugacy graph, such that $c_*^{-1} x^{d\cdot n} c_*$ is rigid. Then $r(c_*^{-1} x^n c_*)=d$. By Lemma~\ref{L:LimitOnPower}, applied to~$x^n$, we have $d=2$. 

The other possibility is that the conjugacy graph of~$x^n$ has more than one vertex. Then Lemma~\ref{L:NoFurtherPowers} implies that the conjugacy graph of any further power $x^{d\cdot n}$ is no larger than the conjugacy graph of~$x^n$.

To summarize the second case, for any element~$y$ conjugate to~$x$ which has some rigid power, we have $r(y)=1$ or $r(y)=2$. In the language of Section~\ref{S:Periodicity} again, all the primitive elements of $\cup_{n\in\mathbb N} SC(x^n)$ lie in $\cup_{n\leqslant 2} SC(x^n)$.

Summarizing the two cases studied above, we see that for any rigid element~$x$, all the primitive elements of $\cup_{n\in\mathbb N} SC(x^n)$ lie in $\cup_{n\leqslant 27} SC(x^n)$. Since, moreover, in the 4-strand braid group, as in all braid groups, roots are unique up to conjugacy \cite{GMRootUnique}, our group $\mathcal{B}_4^*$ satisfies Property~$(*)$ from Section~\ref{S:Periodicity} with $\calR_G=27$, and we conclude that for any rigid element~$x$, the sequence $|SC(x^n|$ is periodic of period at most $lcm(1,2,\ldots,27)$. 
This completes the proof of Main Theorem~\ref{T:main}(b).

{\bf Acknowledgements }
We thank Jean Michel for helpful comments on a first version of this paper.

The second author was supported by the projects PID2022-138719NA-I00 funded by MCIN/AEI/-10.13039/501100011033 and by FEDER, UE.

The third author was partially supported by the projects PID2020-117971GB-C21 and PID2024-157173NB-I00 funded by MCIN/AEI/10.13039/501100011033 and by FEDER, UE.

The fourth author conduced this work within the France 2030 framework program, the Centre Henri Lebesgue  ANR-11-LABX-0020-01

\end{document}